\documentclass[11pt,reqno,a4paper]{amsart}

\usepackage{euscript}
\usepackage{amsmath,amssymb,amsfonts}
\usepackage{fullpage}

\newtheorem{theorem}{Theorem}
\newtheorem{proposition}[theorem]{Proposition}
\newtheorem{lemma}[theorem]{Lemma}

\newtheorem{remark}[theorem]{Remark}
\newtheorem{example}[theorem]{Example}
\newtheorem{hyp}{Hypothesis}

\renewcommand{\epsilon}{\varepsilon}
\renewcommand{\L}{\mathcal{L}}

\def\Id{\text{\rm Id}}

\def\N{\mathbb{N}}
\def\Z{\mathbb{Z}}
\def\R{\mathbb{R}}
\def\B{\mathcal{B}}

\DeclareMathOperator{\esssup}{esssup}
\begin{document}
	
	\title[Statistical stability and linear response]{Statistical stability and linear response for random hyperbolic dynamics}
	
	\begin{abstract}
		We consider families of random products of close-by Anosov diffeomorphisms, and show that statistical stability and linear response hold for the associated families of equivariant and stationary measures. Our analysis relies on the study of the top Oseledets space of a parametrized transfer operator cocycle, as well as ad-hoc abstract perturbation statements. As an application, we show that, when the quenched central limit theorem holds, under the conditions that ensure linear response for our cocycle, the variance in the CLT depends differentiably on the parameter.
	\end{abstract}
	
	\author{Davor Dragi\v cevi\'c}
	\address{Department of Mathematics, University of Rijeka, Croatia}
	\email{ddragicevic@math.uniri.hr}
	
	\author{Julien Sedro}
	\address{Laboratoire de Probabilit\'es, Statistique et Mod\' elisation (LPSM), Sorbonne Universit\' e, Universit\'e de Paris, 4
		Place Jussieu, 75005 Paris, France}
	\email{sedro@lpsm.paris}

	\maketitle
	
	\section{Introduction}\label{sec:I}
	The aim of this paper is to study stability for the families of equivariant and stationary measures associated with a random product of (uniformly) hyperbolic diffeomorphisms. Those stability properties are related to the following question: in the context of non-autonomous dynamics, how does the statistical properties change when one perturbs the dynamics?
	\medskip
	
	More precisely, we will consider here a family of random hyperbolic diffeomorphisms, $T_{\omega,\epsilon}$, acting on some Riemannian manifold $M$ and indexed by $\omega\in\Omega$ and $\epsilon\in I$, where $(\Omega,\mathcal F,\mathbb P)$ is some probability space, and $0\in I\subset\mathbb R$ is some interval. Endowing the probability space with an invertible map $\sigma:\Omega\circlearrowleft$ that is measure-preserving and ergodic, we may form the \emph{random products over $\sigma$}, defined by
	\begin{equation}
		T_{\omega,\epsilon}^n:=T_{\sigma^n\omega,\epsilon}\circ\dots\circ T_{\omega,\epsilon}.
	\end{equation}
	Assuming that this random product admits a \emph{physical equivariant measure}, that is a measure $h_{\omega}^\epsilon$ satisfying the equivariance condition
	\begin{equation}
		T_{\omega,\epsilon}^*h_{\omega}^\epsilon=h_{\sigma\omega}^\epsilon,
	\end{equation}
	and such that $\mathbb P$-a.s, the ergodic basin of $h_{\omega}^\epsilon$ has positive Riemannian volume\footnote{meaning that for $\mathbb P$-a.e $\omega\in\Omega$, the set$\{x\in M,\frac{1}{n}\sum_{k=0}^{n-1}\delta_{T^k_{\omega,\epsilon}x}\longrightarrow h_{\omega}^\epsilon~\text{weakly}\}$ has positive Riemannian volume.}, we ask the following questions: is the map $\epsilon\in I\mapsto h_{\omega}^\epsilon$ continuous at $\epsilon=0$ in some suitable sense ? Is it differentiable ? If so, can one derive an explicit formula for its derivative? 
	\\The first question is the \emph{statistical stability} problem, the last two are called the \emph{linear response} problem. 
	\medskip
	
	Linear response has received extensive attention, in various context: in the deterministic case (which corresponds, in our setting, to the case where $\Omega$ is reduced to a singleton and one considers a smooth family of maps $(T_\epsilon)_{\epsilon\in I}$), expanding maps of the circle \cite{B2} or in higher dimension \cite{Baladibook,S}, piecewise expanding maps of the interval \cite{B1,BS1} or more general unimodal maps \cite{BS2}, intermittent maps \cite{BahSau,BT,K} have been studied. In the setting of hyperbolic dynamics, the problem of linear response was first considered by Ruelle \cite{Ruelle97} for uniformly hyperbolic maps. A different approach, the so-called weak spectral perturbation (or Gou\"ezel-Keller-Liverani) theory, was devised by \cite{GL} (see also \cite{Baladibook}). Finally, we mention the paper \cite{Dol}, where linear response is established for a wide class of partially hyperbolic systems. 
	
	The random case may be divided in two different subcases: the annealed case and the quenched one, the latter of which we will focus on in this paper. The annealed case may be studied by methods very similar to the deterministic one, namely weak spectral perturbation for the associated family of transfer operators, and often enjoy a convenient ``regularization property" (see e.g \cite{GG} or \cite{GS}). We also mention~\cite{BRS}, where the authors deal with  annealed perturbation of uniformly and non-uniformly expanding maps. For annealed perturbation of Anosov diffeomorphisms, very general results were obtained in \cite{GL}. 
	
	The study of the quenched case is more recent, and the literature on the subject is sparse. Indeed, in this situation one cannot use the tools devised in the deterministic or annealed case, as the dynamically relevant objects shift from the spectral data of individual transfer operators to the Lyapunov-Oseledets spectra associated with a cocycle of such transfer operators. For the statistical stability problem in this context, we refer to \cite{B0,BKS,Bo,FGTQ}. Recently, the interesting preprint \cite{C} develops an analogue of the Gou\"ezel-Keller-Liverani theory to study regularity of the exceptional Oseledets spectrum, for quasi-compact cocycles having a dominated splitting, but only up to Lipschitz regularity. This machinery could in principle be applied to our setting, to obtain a result similar to our Theorem \ref{sst}. We observe that, although our result is less general since it only concerns the top Oseledets space, it has the nice property of giving an explicit modulus of continuity, and to have an elementary proof.

	For the response problem, a very general study is presented in \cite{RS}, in the case of a random products of uniformly expanding maps, with a finite or countable number of branches, and in any finite dimension. The idea is to express the equivariant family of measures of the random product as the fixed point of a family of cone-contracting maps that exhibits suitable regularity properties, and to deduce the wanted smoothness of the equivariant measures by some implicit-function like argument. 
	
	We emphasize that the results we present here rely on methods that are quite different from those in the previously discussed paper, as they do not rely on Birkhoff cone contraction techniques. We also remark that, contrary to the expanding case, the use of the Gou\"ezel-Liverani scale of anisotropic spaces (or, for that matter, any of the available scale of anisotropic spaces) limits us to products of nearby (in the $C^{r+1}$ topology) diffeomorphisms.
	
	A few months after the present paper was made available as a preprint, the Gou\"ezel-Keller-Liverani theory for cocycles \cite{C} was further generalized in \cite{CN}, to cover the case of quenched linear, as well as higher order, response. In particular, \cite[Thm 3.6]{CN} generalizes our Theorem \ref{1043} to higher-order Taylor expansions, as remarked in \cite[Rem. 3.8]{CN}. The main idea behind this generalization, namely lifting the cocycle to the so-called Mather operator (to which the deterministic G-K-L theory is then applied), is somehow present in our approach (see e.g. the proof of Proposition \ref{836}), although the latter is independent on weak spectral perturbation theory. We finally remark that, although it should be possible in principle, \cite{CN} does not state any linear response formula.  
	\medskip
	
	Before going any further, we would like to point out a subtle issue, that is peculiar to the the quenched case, and related to the ``suitable sense" for which the question of statistical stability and linear response may be answered. In the deterministic case, this means finding a suitable topology into which the invariant measure will live (e.g. $C^r(\mathbb S^1,\R)$, $r>1$ for the a.c.i.p of an expanding map of the circle, or as a distribution of order one for smooth deformations of unimodal maps, see \cite{BS2}). 
	In the quenched case, one also has to take care of the random parameter $\omega\in\Omega$. There are several natural possibilities:
	the \emph{almost sure} sense (i.e, one studies the a.s regularity of $\epsilon\in I\mapsto h_{\omega}^\epsilon\in\mathcal B$, with $\mathcal B$ a suitable Banach space in the range), the \emph{essentially bounded} sense (where one studies the regularity of $\epsilon\in I\mapsto h_{\omega}^\epsilon\in L^\infty(\Omega,\mathcal B)$), the $L^1$ sense (where the map of interest is $\epsilon\in I\mapsto h_{\omega}^\epsilon\in L^1(\Omega,\mathcal B)$). It is easy to see that the $L^\infty$ sense is the strongest one. Furthermore, given the relation between the equivariant measures and the stationary one, the $L^1$ sense implies asking the questions of stability and response for the stationary measure of the skew-product. 
	However, an ambiguity arises when one considers the ``almost sure" sense : indeed, it may be that the set of random parameters for which certain estimates on the equivariant measure $h_{\omega}^\epsilon$ holds (let us denote it by $\Omega_\epsilon$) depends on $\epsilon$. 
	In this situation, it is not clear whether a statement like ``$h_{\omega}^\epsilon \to h_{\omega}^0$ when $\epsilon\to 0$, $\mathbb P$-a.s.'' has any probabilistic meaning, since it would hold on $\bigcap_{\epsilon\in I}\Omega_\epsilon$, which may be non-measurable set (as the intersection is taken over an uncountable set).
	For this reason, we refrain from considering the ``almost sure" sense for the regularity results we present, and instead focus on the $L^\infty$-sense.
	\medskip
	
	The paper is organized as follows: In Section \ref{sec:main}, after recalling useful properties of the Gou\"ezel-Liverani anisotropic Banach spaces, we present and discuss our set-up (Hypothesis \ref{hyp:reg}), we state our main result (Theorem \ref{thm:main}) as well as a \emph{quenched linear response formula} \eqref{eq:quenchedlinresp}, reminiscent of \cite{RS,Ruelle97}, and give explicit examples of systems to which this setting apply (Section \ref{sec:examples}). In Section \ref{P}, we present abstract theorems on quenched statistical stability (Theorem \ref{sst}) and quenched linear response (Theorem \ref{1043}), applicable in particular to the equivariant measure associated with a (sufficiently) smooth family of Anosov diffeomorphisms cocycles. 
	In Section \ref{PMT}, we give the proof of the main Theorem \ref{thm:main}. In Section \ref{sec:app}, we give various applications of the previous results: first, we remark in Theorem \ref{thm:annealedresp} that Theorem \ref{1043} easily imply response for the stationary measure of the skew-product associated to the cocycle, and that this can be used to establish linear response for a class of deterministic, partially hyperbolic systems. In Section \ref{sec:var}, we prove Theorem \ref{variance} which gives the differentiability  w.r.t the parameter of the variance in the quenched central limit theorem (satisfied by the Birkhoff sum of random observable satisfying certain conditions).
	\\Finally, in Section \ref{sec:other}, we discuss applications of our approach to other type of random hyperbolic systems: random compositions of uniformly expanding maps, or random two-dimensional piecewise hyperbolic maps.
	
	\section{Main Theorem}\label{sec:main}
	
	\subsection{A class of anisotropic Banach spaces introduced by Gou\"{e}zel and  Liverani}\label{Pre}
	The purpose of this subsection is to briefly summarize the main results from~\cite{GL}. More precisely, we recall the properties of the so-called scale of anisotropic Banach spaces, on which the transfer operator associated to a transitive Anosov diffeomorphism has a spectral gap. The discussion we present here will be relevant when building examples under which the abstract results of the present paper are applicable.

	Let $M$ denote a $C^\infty$ compact and connected Riemannian manifold. Furthermore, let $T$ be a transitive  Anosov diffeomorphism on $M$ of class $C^{r+1}$ for $r>1$.  By  $\mathcal L_T$ we will denote the  transfer operator associated to $T$. We recall that the action of $\mathcal L_T$ on smooth functions $h\in C^r(M, \mathbb R)$ is given by
	\[
	\mathcal L_T h=(h\lvert \det (DT)\rvert^{-1})\circ T^{-1}.
	\]
	
	Let us now briefly summarize the  main results from~\cite{GL}.
	Take $p\in \N$, $p\le r$ and $q>0$ such that $p+q<r$. It is proved in~\cite{GL} that  there exist Banach spaces $\mathcal B^{p,q}=(\mathcal B^{p,q}, \lVert \cdot \rVert_{p,q})$ and $\mathcal B^{p-1, q+1}=(\mathcal B^{p-1, q+1}, \lVert \cdot \rVert_{p-1, q+1})$ with the following properties:
	\begin{itemize}
		\item By construction, $C^r(M,\R)$ is dense in $\mathcal B^{i,j}$ for $(i,j)=\{(p,q), (p-1, q+1)\}$;
		\item By \cite[Lemma 2.1]{GL}, $\mathcal B^{p,q}$ can be embedded in $\mathcal B^{p-1, q+1}$ and 
		the unit ball of $\mathcal B^{p,q}$ is relatively compact in $\mathcal B^{p-1, q+1}$;
		\item By~\cite[Proposition 4.1]{GL}, elements of $\mathcal B^{p,q}$ are distributions of order at most $q$;
		\item By \cite[Lemma 3.2]{GL}, multiplication by a $C^{k+q}$ function, $1\le k\le p$, induces a bounded operator on $\mathcal B^{p,q}$. Moreover, the action of a $C^r$ vector field induces a bounded operator from $\mathcal B^{p,q}$ to $\mathcal B^{p-1,q+1}$.
		\item $\mathcal L_T$ acts as a bounded operator on $\mathcal B^{i,j}$ for $(i,j)=\{(p,q), (p-1, q+1)\}$. Moreover, for each $h\in \mathcal B^{i,j}$ and $\varphi \in C^j(M, \R)$, we have that
		\[
		(\mathcal L_T h)(\varphi)=h(\varphi \circ T), 
		\]
		where by $h(\varphi)$ we denote the action of a distribution $h$ on a test function $\varphi$;
		\item By \cite[Lemma 2.2]{GL}, there exist $A>0$ and $a\in (0,1)$ such that
		\[
		\lVert \mathcal L_T^n h\rVert_{p-1, q+1} \le A\lVert  h\rVert_{p-1, q+1}, \ \text{for $n\in \N$ and $h\in \mathcal B^{p-1, q+1}$}
		\]
		and 
		\[
		\lVert \mathcal L_T^n h\rVert_{p,q} \le Aa^n \lVert h\rVert_{p,q}+A\lVert h\rVert_{p-1, q+1}, \ \text{for $n\in \N$ and $h\in \mathcal B^{p, q}$;}
		\]
		\item By \cite[Theorem 2.3]{GL}, $\mathcal L_T$ is a quasi-compact operator on $\mathcal B^{p,q}$ with spectral radius one. Moreover, $1$ is the only eigenvalue of $\mathcal L_T$ on the unit circle. Finally, $1$ is the simple eigenvalue of $\mathcal L_T$ and the corresponding eigenspace is spanned with the unique SRB measure for $T$.
	\end{itemize}
	\subsection{Regularity assumptions}\label{subsec:reghyp}
	In this section, we will state precisely our regularity assumptions and our main theorem. We start by fixing, once and for all, the system of $C^\infty$ coordinates chart to be the $(\psi_i)_{i=1,\dots,N}$, where $\psi_i:(-r_i,r_i)^d\to M$, and such that the $X_i=\psi_i\left((-r_i/2,r_i/2)^d\right)$ cover $M$ are given by the anisotropic norm construction (see \cite{GL}). We also let $\delta$ be the Lebesgue number of the previous cover. Recall the following fact: if $T$ and $S$ are $C^{r+1}$ maps from $M$ to itself, such that $\sup_{x\in M}d_M(Tx,Sx)\le \frac{\delta}{2}$, then one has: for any $i\in\{1,\dots,N\}$, 
	\[\mathcal J_S(i):=\{j\in\{1,\dots,N\},~S(X_i)\cap X_j\not=\emptyset \}=\mathcal J_T(i),\]
	and one may write \[d_{C^{r+1}}(T,S)=\sum_{i=1}^N\sum_{j\in\mathcal J(i)}\|T_{ij}-S_{ij}\|_{C^{r+1}},\]
	where $\mathcal J(i)=\mathcal J_S(i)=\mathcal J_T(i)$ and $T_{ij}=\psi_j^{-1}\circ T\circ\psi_i:(-r_i,r_i)^d\to (-r_j,r_j)^d$ is a map between open sets in $\mathbb R^d$.
	
	For an interval $0\in I\subset \mathbb R$, we consider a $C^s$ mapping $\mathcal T:I\to C^{r+1}(M,M)$, such that $T_0:=\mathcal T(0)(\cdot)$ is a $C^{r+1}$, transitive Anosov diffeomorphism. Up to shrinking $I$, we may and will assume that for all $\epsilon\in I$, $T_\epsilon:=\mathcal T(\epsilon)(\cdot)$ is a $C^{r+1}$ Anosov diffeomorphism, and that $\sup_{\epsilon\in I}d_{C^{r+1}}(T_\epsilon,T_0)\le \frac{\delta}{4}$.
	In particular, for any $i\in\{1,\dots,N\}$, the set
	\[\mathcal J_\epsilon(i):=\{j\in\{1,\dots,N\},~T_\epsilon(X_i)\cap X_j\not=\emptyset \}\]
	is independent of $\epsilon$. We will informally refer to this property by saying that ``the maps $T_\epsilon$ may be read in the same charts".
	
	Consider now a $\Delta>0$, and set $V:=B_{C^s(I,C^{r+1}(M,M))}(\mathcal T,\Delta)$, i.e. we consider a small ball, in $C^s(I,C^{r+1}(M,M))$ topology, centered at $\mathcal T$. Up to shrinking $\Delta$, we may assume that for any $\mathcal S\in V$, any $\epsilon\in I$, $ S_\epsilon:=\mathcal S(\epsilon)(\cdot)$ is an Anosov diffeomorphism, and that $\sup_{\epsilon\in I}d_{C^{r+1}}(T_\epsilon, S_\epsilon)\le \frac{\delta}{4}$. In particular, for any $i\in\{1,\dots,N\}$, the sets
	\[\mathcal J_{\mathcal S}(i):=\{j\in\{1,\dots,N\},~ S_\epsilon(X_i)\cap X_j\not=\emptyset \}\]
	are independent of $\epsilon$ and $\mathcal S$ both (i.e they only depend on $V$).
	\medskip
	
	\noindent We may now describe the type of perturbed cocycle we will consider in the following:
	\begin{hyp}\label{hyp:reg}
		Let $r>4$, $s>1$, and $0\in I\subset\mathbb R$ an interval; let $\mathcal T$ and  $V\subset C^s(I,C^{r+1}(M,M))$ be as previously described. Furthermore, let 
		$(\Omega,\mathcal F, \mathbb P)$ be a probability space, $\sigma \colon \Omega \to \Omega$ an invertible, ergodic $\mathbb P$-preserving transformation  and consider a measurable mapping 
		\[\mathbf{T}\colon \Omega \to V\]  
		Set $T_{\omega,\epsilon}:=\mathbf{T}(\omega)(\epsilon)( \cdot)$, $\omega \in \Omega$ and $\epsilon \in I$.
	\end{hyp}
	Let us make a few comments on this assumption, based on the previous discussion: 
	\begin{itemize}
		\item We choose the neighborhood $V$ sufficiently small so that for any $\omega\in\Omega$, any $\epsilon\in I$, the collection of $T_{\omega,\epsilon}$ can all be read in the same coordinate charts and share the same set of admissible leaves. In particular, one may study their transfer operators on the same anisotropic Banach spaces.
		\item Our assumption is tailored so that for each fixed $\omega\in\Omega$, $\epsilon \mapsto T_{\omega,\epsilon}$ is a smooth curve of Anosov diffeomorphisms, all close-by to a fixed one (namely, $T_{\omega, 0}$).
	\end{itemize} 
	
	We are now in position to formulate our main result. 
	\begin{theorem}\label{thm:main}
		Let $(T_{\omega,\epsilon})_{\omega\in\Omega,\epsilon\in I}$ be a parametrized cocycle of Anosov diffeomorphisms, satisfying Hypothesis \ref{hyp:reg}.
		Then,  by shrinking $I$ if necessary, there exists a triplet of Banach spaces \[\mathcal B_{ss} \subset \mathcal B_s\subset \mathcal B_w, \]
		and for each $\epsilon\in I$ a  unique family $(h_{\omega}^\epsilon)_{\omega\in\Omega}\subset \mathcal B_{ss}$  with the following properties:
		\begin{itemize}
			\item $\omega \mapsto h_\omega^\epsilon$ is measurable for each $\epsilon \in I$;
			\item $h_\omega^\epsilon$ is a probability measure for $\epsilon \in I$ and $\omega \in \Omega$;
			\item $\mathcal L_{\omega, \epsilon} h_\omega^\epsilon=h_{\sigma \omega}^\epsilon$ for $\epsilon \in I$ and  $\omega \in \Omega$, 
			where $\mathcal L_{\omega, \epsilon}$ denotes the transfer operator of $T_{\omega, \epsilon}$;
			\item the map  $I\ni \epsilon \mapsto h_{\omega}^\epsilon \in L^\infty(\Omega,\mathcal B_w)$ is differentiable at $0$, and for $\phi\in C^{r}(M)$, we have that 
			\begin{equation}\label{eq:quenchedlinresp}
				\partial_\epsilon\left[\int_M\phi dh_\omega^\epsilon \right] \bigg{\rvert}_{\epsilon=0}=\sum_{n=0}^\infty\int_M \partial_{\epsilon}\left[\phi\circ T_{\sigma^{-n}\omega}^{(n)}\circ T_{\sigma^{-n-1}\omega,\epsilon}\right] \bigg{\rvert}_{\epsilon=0}dh_{\sigma^{-n-1}\omega},
			\end{equation}
			where $h_\omega:=h_\omega^0$, $\omega \in \Omega$.
		\end{itemize}
		
	\end{theorem}
	
	\subsection{Examples}\label{sec:examples}
	Here are explicit examples of systems satisfying Hypothesis \ref{hyp:reg}.  In all instances, $r>4$ and $s>1$.
	
	\begin{example}\label{EX1}
		Let $q\in\mathbb N$, $\Omega=\{1,\dots,q\}^\mathbb Z$, endowed with a Bernoulli measure. Consider a family $(T_1,\dots,T_q)$ of (close-enough) $C^{r+1}$ Anosov diffeomorphisms of the $d$-dimensional torus $\mathbb{T}^d$, $p:\mathbb T^d\to\mathbb T^d$ be a $C^{r+1}$ mapping and $0\in I\subset \R$ an interval. We set
		\begin{equation}
			\mathbf T(\omega)(\epsilon,x):=T_i(x)+\epsilon p(x),\quad\text{if}~\omega_0=i,
		\end{equation}
		where $x\in\mathbb T^d$, $\epsilon\in I$ and $\omega=(\omega_n)_{n\in\mathbb Z}\in\Omega$. 
	\end{example}
	
	\begin{example}\label{EX2}
		Let $q\in\mathbb N$, $\Omega=\{1,\dots,q\}^\mathbb Z$, endowed with a Bernoulli measure. Consider a $C^{r+1}$ Anosov diffeomorphism $T$ of $\mathbb{T}^d$. Moreover, consider $p_1,\dots,p_q$ $C^{r+1}$ mappings of $\mathbb T^d$ and $0\in I\subset \R$ an interval,  Then we define, for $\epsilon\in I$, $x\in\mathbb T^d$ and $\omega=(\omega_n)_{n\in\mathbb Z}\in\Omega$ the random map
		\begin{equation}
			\mathbf T(\omega)(\epsilon,x)= T(x)+\epsilon p_i(x), \quad \text{if}~\omega_0=i.
		\end{equation} 
	\end{example}
	
	\noindent In both Examples \ref{EX1} and \ref{EX2}, for each $\omega\in\Omega$, $\mathbf T(\omega)\in C^s(I,C^{r+1}(M,M))$. Furthermore, since for each $i\in\{1,\dots,q\}$, the set $\{\mathbf{T}(\omega)=T_i+\epsilon p\}$ (resp. $\{\mathbf{T}(\omega)=T+\epsilon p_i\}$) is the 1-cylinder $\{\omega_0=i\}$, one easily checks that the map is measurable.
	\begin{example}\label{EX3}
		We now consider the following setting: for $\delta>0$, $\omega\in B_{\mathbb{R}^d}(0,\delta)$ (that is randomly chosen w.r.t Lebesgue measure) and $\epsilon_0>0$, we consider a $C^s$-smooth curve of Anosov diffeomorphisms \[ I:=(-\epsilon_0,\epsilon_0) \ni \epsilon \to T_\epsilon\in C^{r+1}(\mathbb T^d,\mathbb T^d).\] 
		Finally, set 
		\[\mathbf T(\omega)(\epsilon,x):= T_\epsilon(x)+\omega, \quad x\in \mathbb T^d.\]
	\end{example}
	\noindent In this last instance, one easily checks that the map $\Omega \ni \omega \mapsto \mathbf T(\omega)\in C^s(I,C^{r+1}(M,M))$ is continuous and thus measurable.
	
	\section{Some abstract results}\label{P}

	\subsection{Quenched  statistical stability for random systems}\label{QSS}
	In this section we will formulate an abstract result regarding the statistical stability of certain random dynamical systems that in particular applies to random hyperbolic dynamics. 
	
	Let $(\Omega, \mathcal F, \mathbb P)$ be a probability space and consider an invertible  transformation $\sigma \colon \Omega \to \Omega$ which preserves $\mathbb P$. Furthermore, let $\mathbb P$ be ergodic.  
	
	Moreover, let $\mathcal B_w=(\mathcal B_w, \lVert \cdot \rVert_w)$ and $\mathcal B_s=(\mathcal B_s, \lVert \cdot \rVert_s)$ be two Banach spaces such that $\mathcal B_s$ is embedded in $\mathcal B_w$ and that $\lVert \cdot \rVert_w \le \lVert \cdot \rVert_s$ on $\mathcal B_s$.
	Suppose that for each $\omega \in \Omega$, $\mathcal L_\omega$ is a bounded operator both on $\mathcal B_w$ and $\mathcal B_s$.  In addition, assume that $\omega \to \mathcal L_\omega$ is strongly measurable on $\mathcal B_s$, i.e. that the map $\omega \mapsto \mathcal L_\omega h$ is measurable for each 
	$h\in \mathcal B_s$.
	For $\omega \in \Omega$
	and $n\in \N$, set
	\[
	\mathcal L_\omega^n:=\mathcal L_{\sigma^{n-1} \omega}\circ \ldots \circ \mathcal L_{\sigma \omega}\circ \mathcal L_\omega. 
	\]
	We consider a fixed, nonzero $\psi\in \mathcal B_s'$, that admits a bounded extension to $\mathcal B_w$ that we still denote by $\psi$,
	and assume that there exist $D, \lambda >0$ such that
	\begin{equation}\label{dec}
		\lVert \mathcal L_\omega^n h\rVert_s \le De^{-\lambda n} \lVert h\rVert_s, 
	\end{equation}
	for $\mathbb P$-a.e. $\omega \in \Omega$, $n\in \N$ and $h\in \mathcal B_s^0$, where 
	\begin{equation}\label{pz} 
		\mathcal B_s^0=\{h\in \mathcal B_s: \psi(h)=0\}.
	\end{equation}
	Obviously, $\mathcal B_s^0$ depends on the choice of $\psi$. However, this dependence has no bearings on our results (see Remark \ref{rem:hyp}), so we do not make it explicit in the notation itself. 
	
	Consider now an interval $I\subset \R$  around $0\in \R$ and suppose that for $\epsilon \in I$, we have  a family $(\mathcal L_{\omega, \epsilon})_{\omega \in \Omega}$ of bounded linear operators on spaces $\mathcal B_s$ and $\mathcal B_w$. Moreover, assume that $\omega \mapsto \mathcal L_{\omega, \epsilon}$ is strongly measurable on $\mathcal B_s$ for each $\epsilon \in I$.
	Analogously  to $\mathcal L_\omega^n$, for $\omega \in \Omega$, $\epsilon \in I$ and $n\in \N$, we define
	\[
	\mathcal L_{\omega, \epsilon}^n:=\mathcal L_{\sigma^{n-1} \omega, \epsilon}\circ \ldots \circ \mathcal L_{\sigma \omega, \epsilon}\circ \mathcal L_{\omega, \epsilon}. 
	\]
	We set $\mathcal L_{\omega, 0}=\mathcal L_\omega$ and we suppose  that 
	there exist $C>0$, $\lambda_1 \in (0,1)$ and a measurable $\Omega' \subset \Omega$ satisfying $\mathbb P(\Omega')=1$  such that for each $\epsilon \in I$:
	\begin{itemize}
		\item for each $\epsilon \in I$,  $\omega \in \Omega'$, $n\in \N$ and $h\in \mathcal B_s$,
		\begin{equation}\label{1}
			\lVert \mathcal L_{\omega, \epsilon}^nh\rVert_s \le C\lambda_1^n \lVert h\rVert_s+C\lVert h\rVert_w;
		\end{equation}
		\item for each $\epsilon \in I$, $\omega \in \Omega'$ and $h\in \mathcal B_s$,
		\begin{equation}\label{2}
			\lVert (\mathcal L_{\omega, \epsilon} -\mathcal L_\omega)h\rVert_w\le C\lvert \epsilon \rvert \cdot  \lVert h\rVert_s;
		\end{equation}
		\item for each $\epsilon \in I$, $\omega \in \Omega'$ and $n\in \N$,
		\begin{equation}\label{3}
			\lVert \mathcal L_{\omega, \epsilon}^n\rVert_w \le C;
		\end{equation}
		\item for each $\epsilon \in I$, $\omega \in \Omega'$, we have that 
		\begin{equation}\label{4}
			\psi(\mathcal L_{\omega, \epsilon} h)=\psi(h) \quad \text{for each $h\in \mathcal B_s$.}
		\end{equation}
	\end{itemize}
	We can assume without any loss of generality that  $\Omega'$ is contained in a full measure set on which~\eqref{dec} holds.
	\begin{remark}\label{rem:hyp}
		\begin{itemize}
			\item  Observe  that we can assume that  $\Omega'$ is $\sigma$-invariant since  we can replace 
			$\Omega'$ with $\Omega''=\bigcap_{k\in \Z}\sigma^k (\Omega')$  which is clearly $\sigma$-invariant and also satisfies $\mathbb P(\Omega'')=1$. Therefore, from now on we will assume that $\Omega'$ is $\sigma$-invariant. 
			\item We note that we can deal with the  more general situation when $\Omega'$ is allowed to depend on $\epsilon$. However, since the current framework is sufficient for applications we have in mind and for the case of simplicity, we will not explicitly deal with this case.
			\item The fact that almost every $\L_{\omega,\epsilon}$ share a left eigenvector is the reason why the dependence on $\psi$ of the space $\mathcal B_s^0$ has no consequence for us. In our examples, $\psi$ will be $\psi(h):=h(1)$ for a finite-order distribution $h$ (and where $1$ denotes the constant test function). 
		\end{itemize}
	\end{remark}
	
	We first show that the above assumptions imply that all the perturbed cocycles $(\mathcal L_{\omega, \epsilon})_{\omega \in \Omega}$ also satisfy the condition of the type~\eqref{dec} whenever $\lvert \epsilon \rvert$ is sufficiently small. More precisely, we have the following auxiliary result.
	\begin{proposition}\label{922}
		There exist $\epsilon_0, D'>0$ and $\lambda' >0$ such that
		\begin{equation}\label{dec1}
			\lVert \mathcal L_{\omega, \epsilon}^n h\rVert_s \le D'e^{-\lambda' n} \lVert h\rVert_s, 
		\end{equation}
		for $\epsilon \in I$ satisfying $\lvert \epsilon \rvert \le \epsilon_0$,  $\omega \in \Omega'$, $n\in \N$ and $h\in \mathcal B_s^0$.
	\end{proposition}
	
	\begin{proof}
		Let $\epsilon_0>0$ be such that
		\begin{equation}\label{nw}
			\frac{C^4}{1-\lambda_1}\epsilon_0<1/2,
		\end{equation}
		and take an arbitrary $\epsilon \in I$ satisfying $\lvert \epsilon \rvert \le \epsilon_0$.
		
		Since
		\[
		\mathcal L_{\omega, \epsilon}^n-\mathcal L_\omega^n=\sum_{k=1}^n  \mathcal L_{\sigma^k \omega, \epsilon}^{ n-k}(\mathcal L_{\sigma^{k-1} \omega, \epsilon}-\mathcal L_{\sigma^{k-1} \omega})\mathcal L_\omega^{k-1},
		\]
		it follows from~\eqref{1}, \eqref{2} and~\eqref{3} that
		\[
		\begin{split}
			\lVert (\mathcal L_{\omega, \epsilon}^n-\mathcal L_\omega^n) h\rVert_w &\le \sum_{k=1}^n \lVert \mathcal L_{\sigma^k \omega, \epsilon}^{ n-k}(\mathcal L_{\sigma^{k-1} \omega, \epsilon}-\mathcal L_{\sigma^{k-1} \omega})\mathcal L_\omega^{k-1}h\rVert_w\\
			&\le C\sum_{k=1}^n \lVert (\mathcal L_{\sigma^{k-1} \omega, \epsilon}-\mathcal L_{\sigma^{k-1} \omega})\mathcal L_\omega^{k-1}h\rVert_w\\
			&\le C^2 \lvert \epsilon \rvert \sum_{k=1}^n \lVert \mathcal L_\omega^{k-1}h\rVert_s\\
			&\le C^2 \lvert \epsilon \rvert \sum_{k=1}^n(C\lambda_1^{k-1}\lVert h\rVert_s+C\lVert h\rVert_w) \\
			&\le C^3 \lvert \epsilon \rvert \bigg{(}\frac{1}{1-\lambda_1} \lVert h\rVert_s+n \lVert h\rVert_w \bigg{)}, 
		\end{split}
		\]
		and thus
		\begin{equation}\label{aux}
			\lVert (\mathcal L_{\omega, \epsilon}^n-\mathcal L_\omega^n) h\rVert_w \le  C^3 \lvert \epsilon \rvert \bigg{(}\frac{1}{1-\lambda_1} \lVert h\rVert_s+n \lVert h\rVert_w \bigg{)},
		\end{equation}
		for  $n\in \N$, $\omega \in \Omega'$ and $h\in \mathcal B_s$. Thus, \eqref{dec}, \eqref{1} and~\eqref{aux} imply that 
		\[
		\begin{split}
			\lVert \mathcal L_{\omega, \epsilon}^{ n+m} h\rVert_s &=\lVert \mathcal L_{\sigma^m \omega, \epsilon}^n \mathcal L_{\omega, \epsilon}^m h\rVert_s  \\
			&\le C\lambda_1^n  \lVert \mathcal L_{\omega, \epsilon}^m h\rVert_s +C\lVert \mathcal L_{\omega, \epsilon}^m h\rVert_w\\
			&\le C\lambda_1^n (C\lambda_1^m \lVert h\rVert_s+C\lVert h\rVert_w)+C(\lVert \mathcal L_\omega^{ m} h\rVert_w+\lVert (\mathcal L_{\omega, \epsilon}^m-\mathcal L_\omega^m) h\rVert_w)\\
			&\le C^2\lambda_1^{n+m}\lVert h\rVert_s+C^2 \lambda_1^n \lVert h\rVert_s+CDe^{-\lambda m} \lVert h\rVert_s+C^4 \lvert \epsilon \rvert \bigg{(}\frac{1}{1-\lambda_1} +m \bigg{)} \lVert h\rVert_s,
		\end{split}
		\]
		for $n, m\in \N$, $\omega \in \Omega'$ and $h\in \mathcal B_s^0$.  Hence (recall also~\eqref{nw}), we can find (by decreasing $\epsilon_0$ if necessary)  $a\in (0, 1)$ and $N_0\in \N$ (independent of $\epsilon$ and $\omega$) such that 
		\begin{equation}\label{aux2}
			\lVert \mathcal L_{\omega, \epsilon}^{ N_0}h \rVert_s\le a \lVert h\rVert_s,
		\end{equation}
		for  $\omega \in \Omega'$ and $h\in \mathcal B_s^0$. 
		
		On the other hand, it follows  readily from~\eqref{1} that
		\begin{equation}\label{aux3}
			\lVert \mathcal L_{\omega, \epsilon}^n \rVert_s \le 2C \quad \text{for $n\in \N$ and  $\omega \in \Omega'$.}
		\end{equation}
		Take now an arbitrary $n\in \N$ and write it as $n=mN_0+k$ for $m, k \in \N \cup \{0\}$, $0\le k<N_0$.  It follows from~\eqref{aux2} and~\eqref{aux3} that 
		\[
		\begin{split}
			\lVert \mathcal L_{\omega, \epsilon}^n h\rVert_s =\lVert \mathcal L_{\omega, \epsilon}^{mN_0+k}h\rVert_s &\le 2C a^m \lVert h\rVert_s  \\
			&=2Ce^{-m \log a^{-1}}\lVert h\rVert_s  \\
			&=2Ce^{\frac{k}{N_0}\log a^{-1}}e^{-\frac{n}{N_0}\log a^{-1}}\lVert h\rVert_s \\
			&\le 2Ce^{\log a^{-1}}e^{-\frac{n}{N_0}\log a^{-1}}\lVert h\rVert_s,
		\end{split}
		\]
		for $\omega \in \Omega'$, $n\in \N$, $h\in \mathcal B_s^0$. We conclude that~\eqref{dec1}  holds with 
		\[ \lambda'=\log a^{-1}/N_0>0 \quad  \text{and} \quad   D'=2Ce^{\log a^{-1}}>0,\] 
		which are independent on $\epsilon $. The proof of the proposition is completed. 
	\end{proof}
	We are now in position to establish the existence of a random fixed point for the cocycle $(\mathcal L_{\omega, \epsilon})_{\omega \in \Omega}$ whenever $\lvert \epsilon \rvert \le \epsilon_0$.
	\begin{proposition}\label{836}
		For each $\epsilon \in I$ satisfying $\lvert \epsilon \rvert \le \epsilon_0$, there exists a unique family $(h_\omega^\epsilon)_{\omega \in \Omega'} \subset \mathcal B_s$ such that:
		\begin{itemize}
			\item $\omega \mapsto h_\omega^\epsilon$ is measurable and bounded, i.e. 
			\begin{equation}\label{x2}
				\sup_{\omega \in \Omega'} \lVert h_\omega^\epsilon \rVert_s<\infty;
			\end{equation}
			\item for $\omega \in \Omega'$,
			\begin{equation}\label{x3}
				\psi(h_\omega^\epsilon) =1;
			\end{equation}
			\item for  $\omega \in \Omega'$,
			\begin{equation}\label{x1}
				\mathcal L_{\omega, \epsilon} h_\omega^\epsilon =h_{\sigma \omega}^\epsilon.
			\end{equation}
		\end{itemize}
	\end{proposition}
	
	\begin{proof}
		Let  $\mathcal Y$ denote the set of all measurable functions $v\colon \Omega' \to \mathcal B_s$ such that 
		\[
		\lVert v\rVert_\infty=\sup_{\omega \in \Omega'} \lVert v(\omega)\rVert_s<\infty.
		\]
		Then, $(\mathcal Y, \lVert \cdot \rVert_\infty)$ is a Banach space.   Set
		\[
		\mathcal Z:=\{v\in \mathcal Y: \psi(v(\omega))=1 \ \text{for $\omega \in \Omega'$} \}.
		\]
		Observe that $\mathcal Z$ is nonempty. Indeed, since $\psi$ is nonzero, there exists $g\in \mathcal B_s$ such that $\psi(g)=1$. Set $v_0\colon \Omega' \to \mathcal B_s$ by $v_0(\omega)=g$ for $\omega \in \Omega'$. Then, $v_0\in \mathcal Z$. 
		We claim that $\mathcal Z$ is a closed subset of $\mathcal Y$. Indeed, let $(v_n)_n$ be a sequence in $\mathcal Z$ that converges to some $v\in \mathcal Y$. Then, we have that 
		\[
		\lvert \psi(v_n(\omega))-\psi(v(\omega))\rvert \le \lVert \psi \rVert_s\cdot \lVert v_n(\omega)-v(\omega)\rVert_s\le \lVert \psi \rVert_s \cdot \lVert v_n-v\rVert_\infty,
		\]
		for $n\in \N$ and  $\omega \in \Omega'$, where $\|\psi\|_s$ denotes the norm of $\psi \in \mathcal B_s'$. Hence, $\psi(v(\omega))=1$ for  $\omega \in \Omega'$ and thus $v\in \mathcal Z$. 
		
		For $\lvert \epsilon \rvert \le \epsilon_0$, we define a linear operator $\mathbb L^\epsilon \colon \mathcal Y \to \mathcal Y$ by
		\[
		(\mathbb L^\epsilon v)(\omega)=\mathcal L_{\sigma^{-1} \omega, \epsilon} v(\sigma^{-1}\omega), \quad \omega \in \Omega'. 
		\]
		It follows from~\eqref{aux3} (together with our assumption that $\omega \mapsto \mathcal L_{\omega, \epsilon}$ is strongly measurable on $\mathcal B_s$ for each $\epsilon$) that $\mathbb L^\epsilon$ is a well-defined and  bounded operator. 
		Moreover, $\mathbb L^\epsilon \mathcal Z\subset \mathcal Z$. Indeed, for each $v\in \mathcal Z$ we have  (using~\eqref{4}) that 
		\[
		\psi((\mathbb L^\epsilon v)(\omega))=\psi(\mathcal L_{\sigma^{-1} \omega, \epsilon} v(\sigma^{-1}\omega))=\psi(v(\sigma^{-1}\omega))=1,
		\]
		for  $\omega \in \Omega'$. Thus, $\mathbb L^\epsilon v \in \mathcal Z$.
		
		Let us now choose $N\in \N$ such that $D'e^{-\lambda'N}<1$.  It follows from~\eqref{dec1} that 
		\[
		\begin{split}
			\lVert (\mathbb L^\epsilon)^Nv_1-(\mathbb L^\epsilon)^Nv_2\rVert_\infty &=\sup_{\omega \in \Omega'} \lVert \mathcal L_{\sigma^{-N} \omega, \epsilon}^N (v_1(\sigma^{-N} \omega)-v_2(\sigma^{-N} \omega))\rVert_s \\
			&\le D'e^{-\lambda'N}\sup_{\omega \in \Omega'} \lVert v_1(\sigma^{-N} \omega)-v_2(\sigma^{-N} \omega)\rVert_s \\
			&\le D'e^{-\lambda' N} \lVert v_1-v_2\rVert_\infty, 
		\end{split}
		\]
		for $\lvert \epsilon \rvert \le \epsilon_0$ and $v_1, v_2\in \mathcal Z$.  Hence,  $(\mathbb L^\epsilon)^N$  is a contraction on $\mathcal Z$ and therefore, $\mathbb L^\epsilon$   has a unique fixed point $v^\epsilon \in \mathcal Z$. 
		Thus,  the family  $(h_\omega^\epsilon)_{\omega \in \Omega'}$ defined by $h_\omega^\epsilon :=v^\epsilon (\omega)$ satisfies~\eqref{x2}, \eqref{x3} and~\eqref{x1}.
		
		In order to establish the uniqueness, it is sufficient to note that each family $(h_\omega^\epsilon)_{\omega \in \Omega'}$ satisfying~\eqref{x2}, \eqref{x3} and~\eqref{x1} gives rise to a fixed point of $\mathbb L^\epsilon$ in $\mathcal Z$,  which is unique. The proof of the proposition is completed. 
	\end{proof}
	Set
	\[
	h_\omega :=h_\omega^0 \quad \omega \in \Omega'. 
	\]
	The following is our statistical stability result.
	\begin{theorem}\label{sst}
		Let $\epsilon\in I$, $|\epsilon|\le \epsilon_0$. Then
		\begin{equation}\label{sb}
			\sup_{\omega \in \Omega'}	\lVert h_\omega^{\epsilon} -h_\omega \rVert_{w} \le C|\epsilon||\log(|\epsilon|)|,
		\end{equation}
		where $C>0$ is independent on $\epsilon$.
	\end{theorem}
	
	Before we establish Theorem~\ref{sst}, we need the following auxiliary result. Let $h^\epsilon$ denote the family $(h_\omega^\epsilon)_{\omega \in \Omega}$ given by Proposition~\ref{836}.
	\begin{lemma}\label{prop:905}
		We have that
		\begin{equation}\label{905}
			\sup_{\lvert \epsilon \rvert \le \epsilon_0}\sup_{\omega \in \Omega'} \lVert h_\omega^\epsilon \rVert_s < \infty. 
		\end{equation}
	\end{lemma}

	\begin{proof}
		We will use the same notation as in the proof of Proposition~\ref{836}. Take an arbitrary  $u\in \mathcal Z$.  It follows from Banach's contraction principle that 
		\[
		h^\epsilon =\lim_{k\to \infty} ( \mathbb L^\epsilon)^{kN} u,
		\]
		for $\lvert \epsilon \rvert \le \epsilon_0$.  Fix now any $\epsilon$ such that $\lvert \epsilon \rvert \le \epsilon_0$. There exists $k_0\in \N$ such that
		\[
		\lVert h^\epsilon -( \mathbb L^\epsilon)^{k_0N} u \rVert_\infty <1.
		\]
		Hence,  using~\eqref{1} we have that
		\[
		\lVert h^\epsilon \rVert_\infty \le 1+ \lVert ( \mathbb L^\epsilon)^{k_0N} u \rVert_\infty \le 2C\lVert u\rVert_\infty+1,
		\]
		which readily implies the conclusion of the lemma. 
	\end{proof}
	We are now in  a position to prove Theorem~\ref{sst}.
	\begin{proof}[Proof of Theorem~\ref{sst}]
		Take an arbitrary $\epsilon \in I$ such that $|\epsilon|\le \epsilon_0$. 
		Observe that
		\begin{equation}\label{tde}
			\begin{split}
				\lVert h_\omega^{\epsilon} -h_\omega \rVert_w &=\lVert \mathcal L_{\sigma^{-n} \omega, \epsilon}^n h_{\sigma^{-n} \omega}^{\epsilon} -\mathcal L_{\sigma^{-n} \omega}^n h_{\sigma^{-n} \omega}\rVert_w \\
				&\le \lVert \mathcal L_{\sigma^{-n} \omega, \epsilon}^n h_{\sigma^{-n} \omega}^{\epsilon} -\mathcal L_{\sigma^{-n} \omega}^n h_{\sigma^{-n} \omega}^{\epsilon}\rVert_w+\lVert \mathcal L_{\sigma^{-n} \omega}^n (h_{\sigma^{-n} \omega}^{\epsilon}-h_{\sigma^{-n} \omega})\rVert_w,
			\end{split}
		\end{equation}
		for each $n\in \N$ and  $\omega \in \Omega'$.  It follows from~\eqref{dec} and~\eqref{905} that there exists $\tilde D>0$ such that 
		\begin{equation}\label{928}
			\lVert \mathcal L_{\sigma^{-n} \omega}^n (h_{\sigma^{-n} \omega}^{\epsilon}-h_{\sigma^{-n} \omega})\rVert_w \le \lVert \mathcal L_{\sigma^{-n} \omega}^n (h_{\sigma^{-n} \omega}^{\epsilon}-h_{\sigma^{-n} \omega})\rVert_s \le \tilde De^{-\lambda n},
		\end{equation}
		for $n\in \N$ and $\omega \in \Omega'$.
		
		On the other hand, it follows from~\eqref{1}, \eqref{2} and~\eqref{3} that 
		\[
		\begin{split}
			\lVert \mathcal L_{\sigma^{-n} \omega, \epsilon}^n h_{\sigma^{-n} \omega}^{\epsilon} -\mathcal L_{\sigma^{-n} \omega}^n h_{\sigma^{-n} \omega}^{\epsilon}\rVert_w &\le \sum_{j=1}^n  \lVert \mathcal L_{\sigma^{-n+j} \omega}^{n-j} (\mathcal L_{\sigma^{-n+j-1}\omega}-\mathcal L_{\sigma^{-n+j-1}\omega, \epsilon})\mathcal L_{\sigma^{-n} \omega, \epsilon}^{ j-1}h_{\sigma^{-n} \omega}^{\epsilon} \rVert_w\\
			&\le C\sum_{j=1}^n \lVert  (\mathcal L_{\sigma^{-n+j-1}\omega}-\mathcal L_{\sigma^{-n+j-1}\omega, \epsilon})\mathcal L_{\sigma^{-n} \omega, \epsilon}^{ j-1}h_{\sigma^{-n} \omega}^{\epsilon}\rVert_w\\
			&\le C^2\lvert \epsilon \rvert \sum_{j=1}^n \lVert  \mathcal L_{\sigma^{-n} \omega, \epsilon}^{ j-1}h_{\sigma^{-n} \omega}^{\epsilon} \rVert_s\\
			&\le 2nC^3\lvert \epsilon \rvert \cdot \lVert  h_{\sigma^{-n} \omega}^{\epsilon}\rVert_s.
		\end{split}
		\]
		Hence, by~\eqref{905} we have that 
		\begin{equation}\label{929}
			\lVert \mathcal L_{\sigma^{-n} \omega, \epsilon}^n h_{\sigma^{-n} \omega}^{\epsilon} -\mathcal L_{\sigma^{-n} \omega}^n h_{\sigma^{-n} \omega}^{\epsilon}\rVert_w \le 2nC^3 \lvert \epsilon\rvert  \sup_{|\epsilon|\le\epsilon_0} \sup_{\omega \in \Omega'} \lVert h_\omega^{\epsilon} \lVert_s,
		\end{equation}
		for $\omega \in \Omega'$ and $n\in \N$. We conclude from~\eqref{tde}, \eqref{928} and~\eqref{929} that
		\[
		\begin{split}
			\sup_{\omega \in \Omega'} \lVert h_\omega^{\epsilon} -h_\omega \rVert_w &\le 2nC^3 \lvert \epsilon \rvert  \sup_{|\epsilon|\le\epsilon_0} \sup_{\omega \in \Omega'} \lVert h_\omega^{\epsilon} \lVert_s+\tilde De^{-\lambda n},
		\end{split}
		\]
		for  $n\in \N$. Taking $n=\big{\lfloor} \frac{\lvert \log (\lvert \epsilon \rvert) \rvert}{\lambda}\big{\rfloor}$, we conclude that~\eqref{sb} holds.
		
	\end{proof}

	\subsection{Quenched linear response for random dynamics}\label{T42}
	Observe that Theorem~\ref{sst} gives the continuity (in the appropriate sense) of the map $\epsilon \mapsto (h_\omega^\epsilon)_{\omega \in \Omega}$ in $\epsilon=0$. We are now concerned with formulating sufficient conditions under which the same map is differentiable in $\epsilon=0$.
	
	Besides requiring the existence of spaces $\mathcal B_w$ and $\mathcal B_s$ as in Subsection~\ref{QSS}, we also require the existence of a third space $\mathcal B_{ss}=(\mathcal B_{ss}, \lVert \cdot \rVert_{ss})$ that can be embedded in $\mathcal B_s$
	and such that $\lVert \cdot \rVert_s \le \lVert \cdot \rVert_{ss}$ on $\mathcal B_{ss}$. As in Subsection~\ref{QSS}, we assume that $\psi$ is a nonzero functional on $\mathcal B_s$, and we shall also assume that it admits a bounded extension to $\mathcal B_w$. We still denote its restriction (resp. extension) to $\mathcal B_{ss}$ (resp. $\mathcal B_w$) by $\psi$.
	Furthermore, we let $(\mathcal L_{\omega,\epsilon})_{\omega \in \Omega,\epsilon\in I}$ be a family such that each
	$\mathcal L_{\omega,\epsilon}$ is a bounded operator on each of those three spaces.  In addition, suppose that $\omega \mapsto \mathcal L_{\omega, \epsilon}$ is strongly measurable on both  $\mathcal B_s$ and $\mathcal B_{ss}$ for each $\epsilon \in I$.
	
	Besides~\eqref{dec}, we also require that 
	\begin{equation}\label{1011}
		\lVert \mathcal L_\omega^n h\rVert_{ss} \le De^{-\lambda n} \lVert h\rVert_{ss}, 
	\end{equation}
	for $\mathbb P$-a.e. $\omega \in \Omega$, $n\in \N$ and $h\in \mathcal B_{ss}^0$, where
	\[
	\mathcal B_{ss}^0=\{h\in \mathcal B_{ss}: \psi(h)=0\}.
	\]
	We define $\mathcal B_s^0$ and $\mathcal B_w^0$ in a similar manner. In particular, $\mathcal B_s^0$ is the same as in~\eqref{pz}.

	In addition, we also assume that there exist $C>0$, $\lambda_1 \in (0, 1)$ and   a measurable  $\Omega' \subset \Omega$ with the property that $\mathbb P(\Omega')=1$ and:
	\begin{itemize}
		\item for each $\epsilon \in I$,  $\omega \in \Omega'$, $n\in \N$ and $h\in \mathcal B_s$, \eqref{1} holds;
		\item for each $\epsilon \in I$, $\omega \in \Omega'$ and $h\in \mathcal B_s$, \eqref{2} holds; 
		\item for each $\epsilon \in I$, $\omega \in \Omega'$ and $n\in \N$, \eqref{3} holds;
		
		\item  for each $\epsilon \in I$,  $\omega \in \Omega'$, $n\in \N$ and $h\in \mathcal B_{ss}$,
		\begin{equation}\label{1x}
			\lVert \mathcal L_{\omega, \epsilon}^n h\rVert_{ss} \le C\lambda_1^n \lVert h\rVert_{ss}+C\lVert h\rVert_s;
		\end{equation}
		\item for each $\epsilon \in I$,  $\omega \in \Omega'$ and $h\in \mathcal B_{ss}$,
		\begin{equation}\label{2x}
			\lVert (\mathcal L_{\omega, \epsilon} -\mathcal L_\omega)h\rVert_{s}\le C\lvert \epsilon \rvert \lVert h\rVert_{ss};
		\end{equation}
		\item for each $\epsilon \in I$ and $\omega \in \Omega'$, we have that for $h\in \mathcal B_{s}$ (and thus also for $h\in \mathcal B_{ss}$)
		\begin{equation}\label{4x}
			\psi(\mathcal L_{\omega, \epsilon} h)=\psi(h).
		\end{equation}
	\end{itemize}
	As before we can assume that $\Omega'$ is contained in a full-measure set on which~\eqref{dec} and~\eqref{1011} hold and  that $\Omega'$ is $\sigma$-invariant.
	
	The following is a direct consequence of Proposition~\ref{922} (applied for the pairs $(\mathcal B_s, \mathcal B_{ss})$ and $(\mathcal B_w, \mathcal B_s)$).
	\begin{lemma}\label{2:56}
		There exist $\epsilon_0, D'>0$ and $\lambda' >0$ such that for $\epsilon \in I$ satisfying $\lvert \epsilon \rvert \le \epsilon_0$,  $\omega \in \Omega'$, $n\in \N$ we have that
		\begin{equation}\label{dec1x}
			\lVert \mathcal L_{\omega,\epsilon}^n h\rVert_{ss} \le D'e^{-\lambda' n} \lVert h\rVert_{ss} \quad \text{for $h\in \mathcal B_{ss}^0$,} 
		\end{equation}
		and
		\begin{equation}\label{dec1x2}
			\lVert \mathcal L_{\omega,\epsilon}^n h\rVert_{s} \le D'e^{-\lambda' n} \lVert h\rVert_{s} \quad \text{for $h\in \mathcal B_{s}^0$.} 
		\end{equation}
	\end{lemma}
	By applying  Proposition~\ref{836} for $\mathcal B_{ss}$ instead of $\mathcal B_s$, we deduce the following result. 
	\begin{proposition}\label{836x}
		For each $\epsilon$ satisfying $\lvert \epsilon \rvert \le \epsilon_0$, there exists a unique family $(h_\omega^\epsilon)_{\omega \in \Omega'} \subset \mathcal B_{ss}$ such that:
		\begin{itemize}
			\item $\omega \mapsto h_\omega^\epsilon$ is measurable and  bounded, i.e.
			\begin{equation}\label{x2x}
				\sup_{\omega \in \Omega'} \lVert h_\omega^\epsilon \rVert_{ss}<\infty;
			\end{equation}
			\item for  $\omega \in \Omega'$,
			\begin{equation}\label{x3x}
				\psi(h_\omega^\epsilon )=1;
			\end{equation}
			\item for $\omega \in \Omega'$,
			\begin{equation}\label{x1x}
				\mathcal L_{\omega, \epsilon} h_\omega^\epsilon =h_{\sigma \omega}^\epsilon.
			\end{equation}
		\end{itemize}
	\end{proposition}
	
	Let us now introduce some additional assumptions. We will suppose that for $\omega \in \Omega'$, there exists a bounded linear operator $\hat{\mathcal L}_\omega \colon \mathcal B_{ss} \to \mathcal B_s$, admitting a bounded extension (which will also be denoted by $\hat\L_{\omega}$) from $\mathcal B_s$ to $\mathcal B_w$, and such that:
	
	\begin{equation}\label{qx}
		\left \{
		\begin{aligned}
			\sup_{\omega \in \Omega'} \lVert \hat{\mathcal L}_\omega\rVert_{\mathcal B_{ss} \to \mathcal B_s}&<\infty,\\
			\sup_{\omega \in \Omega'} \lVert \hat{\mathcal L}_\omega\rVert_{\mathcal B_{s} \to \mathcal B_w}&<\infty,
		\end{aligned}
		\right.
	\end{equation}
	and  we suppose that there is a function $\alpha:I\to\mathbb R_+$, $\lim_{\epsilon \to 0} \alpha(\epsilon)=0$ such that for $\omega \in \Omega'$, 
	\begin{equation}\label{lim}
		\bigg{\lVert } \frac{1}{\epsilon}(\mathcal L_{\omega, \epsilon}-\mathcal L_\omega)h-\hat{\mathcal L}_\omega h \bigg{\rVert}_{w}\le \alpha(\epsilon)\|h\|_{ss} \quad \text{for $h\in \mathcal B_{ss}$ and $\epsilon \in I\setminus \{0\}$.}
	\end{equation}
	
	We emphasize that the inequality \eqref{lim} only holds in $\mathcal{B}_w$-topology. Obviously, $\hat{\mathcal L}_\omega \mathcal B_{ss}^0\subset \mathcal B_s^0$, for  $\omega \in \Omega'$, but it also follows from \eqref{lim} and boundedness of $\psi$ on $\mathcal B_w$ that $\hat\L_\omega:\mathcal B_{ss}\to\mathcal B_s^0$.
	
	Finally, we assume that for $\omega \in \Omega'$ and every $n\in \N$, 
	\begin{equation}\label{3x}
		\lVert \mathcal L_{\omega}^ nh\rVert_{w}\le D'e^{-\lambda'n}\|h\|_w \quad \text{for $h\in \mathcal B_{w}^0$.}
	\end{equation}
	
	We continue to denote $h_\omega^0$ simply by $h_\omega$. For $\omega \in \Omega'$, set 
	\begin{equation}\label{EQ2}
		\hat{h}_\omega:=\sum_{j=0}^\infty \mathcal L_{\sigma^{-j}\omega}^j \hat{\mathcal L}_{\sigma^{-(j+1)} \omega}h_{\sigma^{-(j+1)}\omega}.
	\end{equation}
	It follows from~\eqref{dec}, \eqref{x2x}, \eqref{qx} and the previous discussion that $\hat{h}_\omega \in \mathcal B_s^0$ for  $\omega \in \Omega'$. In addition,
	\begin{equation}\label{sup0}
		\sup_{\omega \in \Omega'}\lVert \hat h_\omega\rVert_s <\infty. 
	\end{equation}
	
	The following is our linear response result.
	\begin{theorem}\label{1043}
		We have that 
		\begin{equation}\label{eq:linresp}
			\lim_{\epsilon \to 0}\sup_{\omega \in \Omega'}\bigg{\lVert} \frac{1}{\epsilon}(h_\omega^{\epsilon}-h_\omega) -\hat{h}_\omega \bigg{\rVert}_{w}=0.
		\end{equation}
		
	\end{theorem}
	
	\begin{proof}
		Let us begin by introducing some auxiliary notation.  Set
		\[
		\tilde{h}^{\epsilon}_\omega:=h_\omega^{\epsilon}-h_\omega  \quad \text{and} \quad \tilde{\mathcal L}_{\omega, \epsilon}:=\mathcal L_{\omega, \epsilon}-\mathcal L_\omega. 
		\]
		It follows easily from~\eqref{x1x} that
		\[
		\tilde{h}_\omega^{\epsilon}-\mathcal L_{\sigma^{-1} \omega} \tilde{h}_{\sigma^{-1} \omega}^{\epsilon} =\tilde{\mathcal L}_{\sigma^{-1} \omega, \epsilon} h_{\sigma^{-1} \omega}^\epsilon,
		\]
		and thus 
		\begin{equation}\label{EQ1}
			\tilde{h}_\omega^{\epsilon}=\sum_{j=0}^\infty \mathcal L_{\sigma^{-j} \omega}^ j\tilde{\mathcal L}_{\sigma^{-(j+1)}\omega, \epsilon} h_{\sigma^{-(j+1)} \omega}^\epsilon,
		\end{equation}
		for $\omega \in \Omega'$.
		By~\eqref{EQ2} and~\eqref{EQ1}, we have that 
		\begin{equation}\label{cent}
			\begin{split}
				\bigg{\lVert} \frac{1}{\epsilon}\tilde{h}_\omega^{\epsilon} -\hat{h}_\omega \bigg{\rVert}_{w} &=\bigg{\lVert}\frac{1}{\epsilon}\sum_{j=0}^\infty \mathcal L_{\sigma^{-j} \omega}^ j\tilde{\mathcal L}_{\sigma^{-(j+1)}\omega, \epsilon} h_{\sigma^{-(j+1)} \omega}^\epsilon-\hat{h}_\omega \bigg{\rVert}_{w} \\
				&\le \bigg{\lVert}\sum_{j=0}^\infty\mathcal L_{\sigma^{-j} \omega}^j \bigg{(} \frac{1}{\epsilon} \tilde{\mathcal L}_{\sigma^{-(j+1)}\omega, \epsilon} -\hat{\mathcal L}_{\sigma^{-(j+1)} \omega} \bigg{)}h_{\sigma^{-(j+1)}\omega}^\epsilon \bigg{\rVert}_{w} \\
				&\phantom{\le}+\bigg{\lVert} \sum_{j=0}^\infty\mathcal L_{\sigma^{-j} \omega}^j\hat{\mathcal L}_{\sigma^{-(j+1)} \omega} \bigg{(}h_{\sigma^{-(j+1)}\omega}^\epsilon-h_{\sigma^{-(j+1)}\omega}\bigg{)} \bigg{\rVert}_{w}.
			\end{split}
		\end{equation}
		By applying Lemma \ref{prop:905}, we have 
		\[\sup_{|\epsilon|\le\epsilon_0}\sup_{\omega \in \Omega'}\|h_\omega^\epsilon\|_{ss}<\infty. \] 
		This, together with  \eqref{lim} and  \eqref{3x}  implies that 
		\begin{equation}\label{357c}
			\begin{split}
				& \bigg{\lVert}\sum_{j=0}^\infty\mathcal L_{\sigma^{-j} \omega}^j \bigg{(} \frac{1}{\epsilon} \tilde{\mathcal L}_{\sigma^{-(j+1)}\omega, \epsilon} -\hat{\mathcal L}_{\sigma^{-(j+1)} \omega} \bigg{)}h_{\sigma^{-(j+1)}\omega}^\epsilon \bigg{\rVert}_{w} \\  
				&\le \sum_{j=0}^\infty  D'e^{-\lambda' j}\bigg{\lVert}  \bigg{(} \frac{1}{\epsilon} \tilde{\mathcal L}_{\sigma^{-(j+1)}\omega, \epsilon} -\hat{\mathcal L}_{\sigma^{-(j+1)} \omega} \bigg{)}h_{\sigma^{-(j+1)}\omega}^\epsilon \bigg{\rVert}_w\\
				&\le \tilde D\alpha(\epsilon)\sup_{|\epsilon|\le\epsilon_0}\sup_{\omega\in\Omega'}\|h_\omega^\epsilon\|_{ss},
			\end{split}
		\end{equation}
		for $\omega \in \Omega'$, where $\tilde D>0$ doesn't depend on $\omega$ and $\epsilon$.
		On the other hand, we have by \eqref{qx} and \eqref{3x} that 
		\[
		\begin{split}
			& \bigg{\lVert} \sum_{j=0}^\infty\mathcal L_{\sigma^{-j} \omega}^j\hat{\mathcal L}_{\sigma^{-(j+1)} \omega} \bigg{(}h_{\sigma^{-(j+1)}\omega}^\epsilon-h_{\sigma^{-(j+1)}\omega}\bigg{)} \bigg{\rVert}_{w} \\
			&\le \sum_{j=0}^{\infty}D'e^{-\lambda'j}\bigg{\lVert} \hat{\mathcal L}_{\sigma^{-(j+1)} \omega} \left(h_{\sigma^{-(j+1)}\omega}^\epsilon-h_{\sigma^{-(j+1)}\omega} \right) \bigg{\rVert}_w \\
			&\le\sup_{\omega \in \Omega'}\|\hat\L_\omega\|_{\mathcal B_s\to \mathcal B_w}\sum_{j=0}^{\infty}D'e^{-\lambda'j}\bigg{\lVert} h_{\sigma^{-(j+1)}\omega}^\epsilon-h_{\sigma^{-(j+1)}\omega}\bigg{\rVert}_s.
		\end{split}
		\]
		
		Now, our assumptions ensure that we may apply Theorem \ref{sst} for the pair $(\mathcal B_{s},\mathcal B_{ss})$. Hence, we get
		\begin{equation}\label{355}
			\bigg{\lVert}\sum_{j=0}^\infty\mathcal L_{\sigma^{-j} \omega}^j\hat{\mathcal L}_{\sigma^{-(j+1)} \omega}\left(h_{\sigma^{-(j+1)}\omega}^\epsilon- h_{\sigma^{-(j+1)}\omega}\right)\bigg{\rVert}_w\le C'|\epsilon||\log|\epsilon|
		\end{equation}
		for  $\omega \in \Omega'$, where $C'>0$ is independent on $\omega$ and $\epsilon$.
		It follows readily from~\eqref{cent}, \eqref{357c} and~\eqref{355} that~\eqref{eq:linresp} holds, which completes the proof of the theorem.
	\end{proof}

	\begin{remark}\label{rem:MET}
		The purpose of this remark is to interpret Theorem~\ref{sst} (as well as Theorem~\ref{1043})  in the context of the multiplicative ergodic theory. 
		In order to do so, we first need to introduce two additional assumptions. Namely, we require that:
		\begin{itemize}
			\item  $\mathcal B_s$ is separable;
			\item the inclusion $\mathcal B_s \hookrightarrow \mathcal B_w$ is compact. 
		\end{itemize}
		By $\Lambda (\epsilon)\in \mathbb R\cup \{-\infty\}$,  we will denote the \emph{largest Lyapunov exponent} of the cocycle $(\mathcal L_{\omega, \epsilon})_{\omega \in \Omega}$, for $\epsilon \in I$. We stress that the existence of $\Lambda (\epsilon)$ is a direct consequence of~\eqref{1} (applied to $n=1$) and the subadditive  ergodic theorem. Moreover, we recall that 
		\[
		\Lambda(\epsilon)=\lim_{n\to \infty} \frac 1 n \log \lVert \mathcal L_{\omega, \epsilon}^n \rVert_s, \quad \text{for $\mathbb P$-a.e $\omega \in \Omega$.}
		\]
		By using~\eqref{1} together with Proposition~\ref{836}, it is easy to show (see the proof of~\cite[Lemma 3.5.]{DFGTV}) that $\Lambda (\epsilon)=0$, for $\epsilon \in I$ with $|\epsilon| \le \epsilon_0$. Moreover, for each such $\epsilon$, the cocycle $(\mathcal L_{\omega, \epsilon})_{\omega \in \Omega}$ is quasi-compact (in the sense of~\cite[Definition 2.7.]{GTQ}). Hence, it follows from the multiplicative ergodic theorem (see~\cite[Theorem A.]{GTQ}) that for each $\epsilon \in I$ with $|\epsilon| \le \epsilon_0$, there exists:
		\begin{itemize} \item  $1\le l=l(\epsilon)\le \infty$ and a sequence of \emph{exceptional
				Lyapunov exponents}
			\[0=\Lambda (\epsilon)=\lambda_1(\epsilon)>\lambda_2(\epsilon)>\ldots>\lambda_l(\epsilon)>\kappa (\epsilon)\] or in the case
			$l=\infty$, 
			\[0=\Lambda (\epsilon)=\lambda_1(\epsilon)>\lambda_2(\epsilon)>\ldots \quad  \text{with
				$\lim_{n\to\infty} \lambda_n(\epsilon)=\kappa (\epsilon)$};
			\]
			\item  a unique measurable Oseledets splitting
			$$\mathcal{B}_s=\left(\bigoplus_{j=1}^l Y_j^\epsilon(\omega)\right)\oplus V^\epsilon(\omega),$$
			where each component of the splitting is equivariant under $\mathcal{L}_{\omega, \epsilon}$, that is,
			$\mathcal L_{\omega, \epsilon}(Y_j^\epsilon(\omega))= Y_j^\epsilon(\sigma\omega)$ and
			$\mathcal L_{\omega, \epsilon}(V^\epsilon(\omega))\subset V^\epsilon(\sigma\omega)$. The subspaces $Y_j^\epsilon(\omega)$ are finite-dimensional and for each $y\in Y_j^\epsilon(\omega)\setminus\{0\}$,
			\[\lim_{n\to\infty}\frac 1n\log\|\mathcal L_{\omega, \epsilon}^n y\|=\lambda_j(\epsilon).\]
			Moreover, for $y\in V(\omega)$, $\lim_{n\to\infty}\frac 1n\log\|\mathcal
			L_{\omega, \epsilon}^n y\|\le \kappa(\epsilon)$.
		\end{itemize}
		
		It follows easily from Proposition~\ref{922} (see the proof of~\cite[Proposition 3.6.]{DFGTV}) that $Y_1^\epsilon (\omega)$ is one-dimensional and is  spanned by $h_\omega^\epsilon$, for each $\epsilon \in I$ such that $|\epsilon| \le \epsilon_0$.
		%
		%
		
		Hence, Theorem~\ref{sst} can be interpreted as a regularity result for the top-Oseledets space of $(\mathcal L_{\omega, \epsilon})_{\omega \in \Omega}$. Namely, it shows that it is continuous in appropriate sense in $\epsilon=0$.  Taking into account that Lyapunov exponents and corresponding Oseledets subspaces represent nonautonomous versions of the classical notions of
		an eigenvalue and the corresponding eigenspace, we conclude that Theorem~\ref{sst} is a natural extension of statistical stability results concerned with deterministic systems.  In a similar manner, Theorem~\ref{1043} can be viewed as a nonautonomous version of the linear response results.
	\end{remark}
	\section{Proof of the Main Theorem}\label{PMT}
	In this section, we prove Theorem \ref{thm:main}, by showing that the assumptions of our abstract Theorems \ref{sst} and \ref{1043} are satisfied. 
	
	We place ourselves in the context of Subsection \ref{subsec:reghyp}: we fix a small enough interval $0\in I\subset \mathbb R$, and we consider a $C^s$ mapping $\mathcal T:I\to C^{r+1}(M,M)$, such that $T_0:=\mathcal T(0)(\cdot)$ is a $C^{r+1}$, transitive Anosov diffeomorphism. 
	\par\noindent We now let $\Delta>0$ and consider $V:=B_{C^s(I,C^{r+1}(M,M))}(\mathcal T,\Delta)$. One has the following lemma:
	\begin{lemma}\label{lemma:controversial}
		There exists $C>0$, depending only on $\mathcal T$ and $\Delta$, such that for any $\mathcal S\in V$, any $\epsilon\in I$,
		\begin{equation}
			d_{C^{r+1}}(S_\epsilon,S_0)\le C|\epsilon|.
		\end{equation}
	\end{lemma}
	\begin{proof}
		From the discussion in Subsection \ref{subsec:reghyp}, it follows that for any $\mathcal S\in V$,
		\[d_{C^{r+1}}(S_\epsilon,S_0)=\sum_{i=1}^N\sum_{j\in\mathcal J(i)}\|S_{ij}(\epsilon,\cdot)-S_{ij}(0,\cdot)\|_{C^{r+1}}.\]
		where we denoted $S_{ij}(\epsilon,\cdot)=\psi_j^{-1}\circ S_\epsilon\circ\psi_i$ for $j\in\mathcal J(i)$. From the mean value theorem, one gets $S_{ij}(\epsilon,\cdot)-S_{ij}(0,\cdot)=\int_0^\epsilon\partial_\epsilon S_{ij}(\eta,\cdot) d\eta$, and hence
		\begin{align*}
			\|S_{ij}(\epsilon,\cdot)-S_{ij}(0,\cdot)\|_{C^{r+1}}&\le\int_0^\epsilon\|\partial_\epsilon S_{ij}(\eta,\cdot)\|_{C^{r+1}}d\eta\\
			&\le C(\mathcal T,\Delta)|\epsilon| 
		\end{align*}
		from which the conclusion follows.
	\end{proof}
	We will consider the following triplet of Banach spaces:
	\begin{equation}\label{Triplet}
		\mathcal B_{ss}=\mathcal B^{3,1}(T_0,M) \hookrightarrow \mathcal B_s=\mathcal B^{2,2}(T_0,M) \hookrightarrow \mathcal B_w=\mathcal B^{1,3}(T_0,M).
	\end{equation}
	We consider a measurable map  $\mathbf{T}:\Omega\to V$, and we write $T_{\omega,\epsilon}=\mathbf{T}(\omega)(\epsilon)(\cdot)$. 
	Finally, we let $\psi$ be defined by $\psi(h)=h(1)$, which is a bounded functional on all three spaces in~\eqref{Triplet}.
	\begin{proof}[Proof of Theorem \ref{thm:main}]
		\par\noindent
		\begin{enumerate}
			\item By Lemma \ref{lemma:controversial} we have, for $\epsilon>0$, that $d_{C^{r+1}}(T_{\omega},T_{\omega,\epsilon})\le C|\epsilon|$, with $C$ independent of $\epsilon$ and $\omega$. Hence,  \cite[Lemma 7.1]{GL} implies that  \eqref{2} and \eqref{2x} hold.
			\item Since $T$ is transitive, the deterministic transfer operator associated to $T$ has a spectral gap on all three spaces $\mathcal B_{ss}, \mathcal B_s$ and $\mathcal B_w$.\footnote{Observe that $\mathcal B_w$ is compactly embedded into $\mathcal B^{0,4}$.}
			Consequently, it follows from~\cite[Proposition 2.10]{CR} that by shrinking $\delta$ is necessary, we have that \eqref{dec}, \eqref{1011} and \eqref{3x} hold. 
			\item The uniform Lasota-Yorke inequalities \eqref{1}, \eqref{3} and \eqref{1x} may be established arguing as in~\cite[Section 3.2]{DFGTV} or~\cite[Section 7]{GL}.
			\item By arguments analogous to those in~\cite[Subsection 3.1.]{DFGTV}, one has that the cocycle $(\mathcal L_{\omega, \epsilon})_{\omega \in \Omega}$ is strongly measurable on $\mathcal B_s$ and $\mathcal B_{ss}$.
		\end{enumerate}
		The previous argument are enough to apply Proposition \ref{836} and Theorem \ref{sst} to our situation, giving us an equivariant family $(h_{\omega}^\epsilon)_{\omega \in \Omega} \subset \mathcal B_{ss}$, that satisfies our statistical stability estimate \eqref{sb} with respect to the norm $\|\cdot \|_{2,2}$. We note that (see~\cite[Proposition 3.3.]{DFGTV}) that for $\epsilon \in I$, $h_{\omega}^\epsilon$ is actually a positive probability  measure on $M$ for $\mathbb P$-a.e. $\omega \in \Omega$.
		
		\medskip
		
		What is now left to do is to establish the existence and required properties of the ``derivative operator''. Following the lines of \cite[Section 9]{GL}, we will systematically abuse notations and ignore coordinates charts.

		Denote by $g_\omega(\epsilon,\cdot):=\frac{1}{|\det(DT_{\omega,\epsilon})|}$ the weight of the transfer operator $\L_{\omega,\epsilon}$. Under our assumptions, when viewed in coordinates, the maps $\epsilon\mapsto g_\omega(\epsilon,\cdot)\in C^r(M,\R^\ast)$ and $\epsilon\mapsto T_{\omega, \epsilon}(\cdot)^{-1}$ are of class $C^s$, $s>1$. In particular, we may, for $\phi\in C^r(M,\R)$, differentiate $\L_{\omega,\epsilon}\phi$ w.r.t $\epsilon$ and obtain:
		\begin{align}
			\partial_\epsilon[\L_{\omega,\epsilon}\phi]&=\L_{\omega,\epsilon}\left(J_{\omega,\epsilon}\phi+v_{\omega,\epsilon}\phi\right)\\
			\partial_\epsilon^2[\L_{\omega,\epsilon}\phi]&=\L_{\omega,\epsilon}\left(J_{\omega,\epsilon}^2\phi+J_{\omega,\epsilon}(v_{\omega,\epsilon}\phi)+v_{\omega,\epsilon}(J_{\omega,\epsilon}\phi)+v_{\omega,\epsilon}(v_{\omega,\epsilon}\phi)+[\partial_\epsilon J_{\omega,\epsilon}]\cdot \phi+\partial_{\epsilon}[v_{\omega,\epsilon}\phi]\right),
		\end{align}
		where
		\begin{align}\label{eq:formaldiff}
			v_{\omega,\epsilon}\phi&:=-D\phi(\cdot) \cdot [DT_{\omega,\epsilon}(\cdot)]^{-1}\cdot\partial_\epsilon T_{\omega}(\epsilon,\cdot)\\
			J_{\omega,\epsilon}&:=\frac{\partial_\epsilon g_\omega(\epsilon,\cdot)+v_{\omega,\epsilon}g_{\omega}(\epsilon,\cdot)}{g_\omega(\epsilon,\cdot)}.
		\end{align}	
		Note that both of the above expressions are, together with their first $s$-derivatives w.r.t. $\epsilon$, in $C^{r-1}(M,\R)$. We will also denote by $v_{\omega,\epsilon}$ the $C^r$ vector field associated with the operator $v_{\omega,\epsilon}$. As noted in Section \ref{Pre}, multiplication by $J_{\omega,\epsilon}$ and the action of $v_{\omega,\epsilon}$ induce bounded operator from $\B^{i,j}$ to itself (resp. $\B^{i,j}$ to $\B^{i-1,j+1}$), where $i+j<r$, and the same goes for their derivatives w.r.t. $\epsilon$.
		\par Furthermore, by our Assumption \ref{hyp:reg}, $J_{\omega,\epsilon}$ and $v_{\omega,\epsilon}$, as well as their derivatives w.r.t. $\epsilon$, are bounded uniformly in $\omega$ and $\epsilon$, i.e
		\begin{align*}
			\max\left(\sup_{\omega\in\Omega}\sup_{\epsilon\in I}\|J_{\omega,\epsilon}\|_{C^{r-1}},\sup_{\omega\in\Omega}\sup_{\epsilon\in I}\|\partial_\epsilon J_{\omega,\epsilon}\|_{C^{r-1}}\right)&<\infty\\
			\max\left(\sup_{\omega\in\Omega}\sup_{\epsilon\in I}\|v_{\omega,\epsilon}\|_{C^r},\sup_{\omega\in\Omega}\sup_{\epsilon\in I}\|\partial_\epsilon v_{\omega,\epsilon}\|_{C^r}\right)&<\infty
		\end{align*}
		For $\phi\in C^r(M,\R)$, set
		\begin{equation}
			\hat\L_\omega\phi:=\partial_\epsilon[\L_{\omega,\epsilon}\phi] \big{\rvert}_{\epsilon=0}=\L_{\omega}\left(J_{\omega,0}\phi+v_{\omega,0}\phi\right).
		\end{equation}
		By our previous discussion,  we conclude that~\eqref{qx} holds.

		On the other hand, using Taylor's formula we conclude that  for $|\epsilon|$ small enough, 
		\begin{equation*}
			\L_{\omega,\epsilon}\phi-\L_{\omega}\phi-\epsilon\hat\L_\omega\phi=\int_0^\epsilon\int_0^\eta \partial_\epsilon^2[\L_{\omega,\epsilon}\phi] \big{\rvert}_{\epsilon=\xi}\, d\xi \, d\eta.
		\end{equation*}	
		By \eqref{eq:formaldiff} and the discussion below, \[\|\partial_\epsilon^2[\L_{\omega,\epsilon}\phi] \big{\rvert}_{\epsilon=\xi}\|_{w}\le C\|\phi\|_{ss},\]
		where $C>0$ independent of both $\omega$ and $\epsilon$. Hence \eqref{lim} is satisfied, and we may apply Theorem \ref{1043}, which gives us that the map $\epsilon\in I\mapsto h_{\omega}^\epsilon\in L^\infty(\Omega,\B_w)$ is differentiable at $\epsilon=0$. Moreover,  \begin{equation}\label{hath}
			\hat h_\omega:=[\partial_\epsilon h_{\omega}^\epsilon] \big{\rvert}_{\epsilon=0}=\sum_{n=0}^\infty\L^{(n)}_{\sigma^{-n}\omega}\hat \L_{\sigma^{-n-1}\omega}h_{\sigma^{-n-1}\omega}.\end{equation}

		To obtain \eqref{eq:quenchedlinresp}, we notice that, by density of smooth functions in $\B^{i,j}$ and \eqref{lim}, $\hat\L_\omega$, as a bounded operator from $\B^{i,j}$ to $\B^{i-1,j+1}$, admits the representation\footnote{In fact, this formula defines a bounded operator from $\mathcal D'_j$ to $\mathcal D'_{j+1}$, but we won't need it.}
		\[(\hat\L_\omega f)(\phi):=f(\partial_\epsilon[\phi\circ T_{\omega,\epsilon}] \rvert_{\epsilon=0}),\]
		for any $f\in \B^{i,j}$ and $\phi\in C^r(M, \R)$. 
		Then, for $\phi\in C^{r}(M,\R)$ we have that 
		\begin{align*}
			\partial_\epsilon\left[\int_M\phi dh_{\omega}^\epsilon\right] \bigg{\rvert}_{\epsilon=0}&=\partial_\epsilon\left[h_{\omega}^\epsilon(\phi)\right] \rvert_{\epsilon=0}\\
			&=\hat h_\omega(\phi)\\
			&=\sum_{n=0}^\infty\L^{(n)}_{\sigma^{-n}\omega}\hat \L_{\sigma^{-n-1}\omega}h_{\sigma^{-n-1}\omega}(\phi)\\
			&=\sum_{n=0}^\infty \hat \L_{\sigma^{-n-1}\omega}h_{\sigma^{-n-1}\omega}(\phi\circ T^{(n)}_{\sigma^{-n}\omega})\\
			&=\sum_{n=0}^\infty h_{\sigma^{-n-1}\omega}\left(\partial_\epsilon\left[\phi\circ T^{(n)}_{\sigma^{-n}\omega}\circ T_{\sigma^{-n-1}\omega,\epsilon}\right] \bigg{\rvert}_{\epsilon=0}\right),
		\end{align*} 
		which gives \eqref{eq:quenchedlinresp}. This completes the proof of Theorem~\ref{thm:main}.
	\end{proof}

	\section{Applications}\label{sec:app}
	In this section, we will present two applications of our main result. Let us assume that the assumptions in Hypothesis~\ref{hyp:reg} hold. 
	We consider the triplet of spaces given by~\eqref{Triplet}.  Furthermore, for $\epsilon \in I$ sufficiently close to $0$, let $(h_\omega^\epsilon)_{\omega \in \Omega}\subset \mathcal B_{ss}$ be as in Section~\ref{PMT}.  By shrinking $I$ if necessary, we can assume that $h_\omega^\epsilon$ exists for $\epsilon \in I$
	and  $\omega \in \Omega$.
	Moreover, recall that  $h_\omega^\epsilon$ is a probability measure on $M$ for  
	$\omega \in \Omega$ (see Section~\ref{PMT}). As before, we write $h_\omega$ instead of $h_\omega^0$.
	\subsection{Annealed linear response for hyperbolic dynamics}
	As a first application, we establish a form of an annealed linear response. 
	
	For $F \in L^\infty(\Omega,C^r(M))$ and $\epsilon\in I$, we set 
	\begin{equation}\label{eq:R}
		R(\epsilon, F)=\int_\Omega \int_M F(\omega,x)\, dh_{\omega}^\epsilon(x) \, d \mathbb P(\omega).
	\end{equation} 
	The following is our annealed linear response result. 
	\begin{theorem}\label{thm:annealedresp}
		The map $R:I\times L^\infty(\Omega,C^r(M))\to\R$ is differentiable at every $(0, F)$, $F \in L^\infty(\Omega,C^r(M))$. Furthermore, one has
		\begin{align}\label{eq:annealedresp}
			\partial_\epsilon[R(\epsilon,F)] \big \rvert_{\epsilon=0}&=\sum_{n=0}^\infty\int_\Omega\int_M \partial_\epsilon \left[F_\omega\circ T^{(n)}_{\sigma^{-n}\omega}\circ T_{\sigma^{-n-1}\omega,\epsilon}\right] \bigg{\rvert}_{\epsilon=0}\, dh_{\sigma^{-n-1}\omega} \, d \mathbb P(\omega).
		\end{align}
	\end{theorem}
	
	\begin{remark}
		The previous result can be interpreted as \emph{linear response for the stationary measure of the skew-product} \[S_\epsilon(\omega,x):=(\sigma\omega,T_{\omega,\epsilon}x),\]
		acting on $\Omega\times M$. Indeed, the stationary measure $\mu_\epsilon$ of this skew-product classically admits the disintegration along fibers
		\[\mu_\epsilon(A\times B)=\int_{A} h^\epsilon_{\omega}(B)\,d\mathbb P(\omega),\]
		for measurable $A\subset \Omega$, $B\subset M$.
		In particular, this justifies the ``annealed" terminology, since in the i.i.d case, the measure defined on $M$ by $\tilde\mu_\epsilon(\cdot)=\mu_\epsilon(\Omega\times\cdot)$ corresponds to the invariant measure of the Markov chain associated with our cocycle. 
		\\We also point out that one may use this interpretation to establish linear response for a class of \emph{deterministic partially hyperbolic skew-products}: let us set $\Omega=\mathbb S^1$, $\mathbb P=\text{Lebesgue}$, and $\sigma(\omega)=\omega+\alpha \mod 1$ for some $\alpha\in\R\backslash\mathbb Q$. Then, consider a family $(T_{\omega,\epsilon})_{\omega\in\mathbb S^1,\epsilon\in I}$ of Anosov diffeomorphisms of $\mathbb T^2$, e.g.
		\[T_{\omega,\epsilon}(x_1,x_2):= 
		\begin{pmatrix}
			2 & 1\\
			1 & 1
		\end{pmatrix}
		\begin{pmatrix}
			x_1\\
			x_2
		\end{pmatrix}
		+
		\begin{pmatrix}
			\omega\\
			\omega
		\end{pmatrix}
		+\epsilon
		\begin{pmatrix}
			\sin 2\pi x_1\\
			\sin2\pi x_2
		\end{pmatrix}
		\]
		This system clearly satisfies our Hypothesis \ref{hyp:reg} (notice that it belongs to the class of Examples \ref{EX3}), and the skew-product $S_\epsilon$ acting on $\mathbb S^1\times\mathbb T^2\simeq \mathbb T^3$ is clearly a partially hyperbolic system (with central direction tangent to the first coordinate), exhibiting linear response by Theorem \ref{thm:annealedresp} and the previous discussion.
	\end{remark}
	
	\begin{proof}[Proof of Theorem \ref{thm:annealedresp}] 
		Fix an arbitrary $F_0 \in L^\infty(\Omega,C^r(M))$. We claim that the derivative of $R$ in $(0, F_0)$ is given by
		\begin{equation}\label{DL}
			DR(0, F_0)(\epsilon, H)=\epsilon\int_\Omega \hat h_\omega (F_0(\omega))\, d\mathbb P(\omega)+\int_\Omega h_\omega( H(\omega))\, d\mathbb P(\omega), 
		\end{equation}
		for $(\epsilon, H)\in \R \times L^\infty(\Omega,C^r(M))$, where $\hat h_\omega$ is given by~\eqref{hath}.  Indeed, observe that 
		\[
		\begin{split}
			& R(\epsilon, F_0+H)-R(0, F_0)-\epsilon\int_\Omega \hat h_\omega (F_0(\omega))\, d\mathbb P(\omega)-\int_\Omega h_\omega( H(\omega))\, d\mathbb P(\omega) \\
			&=\int_\Omega (h_\omega^\epsilon-h_\omega-\epsilon \hat h_\omega )(F_0(\omega))\, d\mathbb P(\omega)+\int_\Omega (h_\omega^\epsilon-h_\omega)(H(\omega))\, d\mathbb P(\omega).
		\end{split}
		\]
		Furthermore, the continuous embedding $\mathcal B^{p,q}\hookrightarrow \mathcal D'_q$ entails that there is $C>0$ (independent on both $\omega$ and $\epsilon$) such that
		\[
		\bigg |\frac{1}{\epsilon} \int_\Omega (h_\omega^\epsilon-h_\omega-\epsilon \hat h_\omega )(F_0(\omega))\, d\mathbb P(\omega) \bigg | \le C \lVert F_0\rVert_{L^\infty(\Omega,C^r(M))} \cdot \sup_{\omega \in \Omega}\bigg{\lVert} \frac{1}{\epsilon}(h_\omega^{\epsilon}-h_\omega) -\hat{h}_\omega \bigg{\rVert}_{w},
		\]
		and thus Theorem~\ref{1043} implies that 
		\[
		\lim_{\epsilon \to 0} \frac{1}{\epsilon} \int_\Omega (h_\omega^\epsilon-h_\omega-\epsilon \hat h_\omega )(F_0(\omega))\, d\mathbb P(\omega)=0.
		\]
		In addition, 
		\[
		\bigg |\int_\Omega (h_\omega^\epsilon-h_\omega)(H(\omega))\, d\mathbb P(\omega)\bigg | \le C\lVert H\rVert_{L^\infty(\Omega,C^r(M))} \cdot \sup_{\omega \in \Omega} \|h_\omega^\epsilon-h_\omega \|_w, 
		\]
		and consequently by applying  Theorem \ref{sst} (for the pair $(\mathcal B_{s},\mathcal B_{w})$), we obtain that 
		\[
		\lim_{(\epsilon,H)\to (0,0)} \frac{1}{\lVert H\rVert_{L^\infty(\Omega,C^0(M))}}\bigg |\int_\Omega (h_\omega^\epsilon-h_\omega)(H(\omega))\, d\mathbb P(\omega)\bigg |=0.
		\]
		Thus, \eqref{DL} holds and the proof of the theorem is completed. In order to establish~\eqref{eq:annealedresp}, one can argue as in  the proof of formula \eqref{eq:quenchedlinresp}.
	\end{proof}
	\subsection{Regularity of the variance in the central limit theorem for random hyperbolic dynamics}\label{sec:var}
	In this section, we provide an application of Theorem~\ref{1043} to the problem of the regularity of the variance (under suitable perturbations) in the quenched version of the central limit theorem for random hyperbolic dynamics. 
	
	Let $F$ be as in the previous subsection. For $\omega \in \Omega$ and $\epsilon \in I$, set
	\[f_{\omega,\epsilon}:= F_\omega- h_{\omega}^\epsilon (F_\omega)=F_\omega -\int_M F_\omega \, dh_\omega^\epsilon.\] 
	Set 
	\begin{equation}\label{def:var}
		\Sigma^2_\epsilon :=\int_\Omega\int_M f_{\omega,\epsilon}^2(x)dh_{\omega}^\epsilon(x)d\mathbb P(\omega)+2\sum_{n=1}^\infty\int_\Omega\int_M f_{\omega,\epsilon}(x)f_{\sigma^n\omega,\epsilon}(T^n_{\omega,\epsilon}x)dh_{\omega}^\epsilon(x)d\mathbb P(\omega).
	\end{equation}
	Observe that $\Sigma^2_\epsilon \ge 0$ and that $\Sigma^2_\epsilon$ doesn't depend on $\omega$. 
	It is proved in~\cite[Theorem B]{DFGTV} that if $\Sigma_\epsilon^2>0$,  the process $(f_{\omega,\epsilon}\circ T^n_{\omega,\epsilon})$ satisfies $\mathbb P$-a.s a quenched central limit theorem. More precisely,  for every bounded and continuous $\phi:\mathbb R\to\mathbb R$ and $\mathbb P$-a.e $\omega\in\Omega$, we have that 
	\[\lim_{n\to\infty}\int \phi\left(\frac{S_n(f_{\omega,\epsilon})}{\sqrt n}\right )  dh_{\omega}^\epsilon=\int\phi d\mathcal N(0,\Sigma^2_\epsilon), \]
	where \[S_n(f_{\omega, \epsilon}):=\sum_{k=0}^{n-1}f_{\sigma^k\omega, \epsilon}\circ T_{\omega,\epsilon}^k,\]
	and $\mathcal N(0, \Sigma^2_\epsilon)$ denotes the normal distribution with parameters $0$ and $\Sigma_\epsilon$.
	Our goal  is to establish the following result. 
	\begin{theorem}\label{variance}
		Under the above assumptions, the map $\epsilon\mapsto\Sigma^2_\epsilon$ is differentiable at $\epsilon=0$.
	\end{theorem}
	We start the proof by making  few remarks related to the map $\epsilon\mapsto (f_{\omega,\epsilon})_{\omega \in \Omega}\in C^r(M)$:
	\begin{itemize}
		\item For each $\epsilon$, $\omega \mapsto f_{\omega, \epsilon}$  is an element of $L^\infty(\Omega,C^r(M))$. Moreover,  by Lemma~\ref{prop:905} we have that 
		\begin{equation}\label{bnd}
			\sup_{|\epsilon|\le \epsilon_0}\esssup_{\omega\in \Omega}\|f_{\omega,\epsilon}\|_{C^r}\leq (1+\sup_{|\epsilon|\le\epsilon_0}\esssup_{\omega\in\Omega}\|h_\omega^\epsilon\|_{ss})\esssup_{\omega\in\Omega}\|F_\omega\|_{C^r}.
		\end{equation}
		\item it is differentiable at $\epsilon=0$. Indeed, we have 
		\[\frac{1}{\epsilon}\left( f_{\omega,\epsilon}-f_\omega\right)= \frac{1}{\epsilon} ( h_\omega-h_\omega^\epsilon) (F_\omega),\]
		which yields
		\begin{equation}\label{TET} 
			\esssup_{\omega\in\Omega}\left|\frac{1}{\epsilon}\left( f_{\omega,\epsilon}-f_\omega\right)+\hat h_\omega (F_\omega)\right|\to 0,  \end{equation}
		as $\epsilon\to 0$, via Theorem \ref{1043}. Here, we write $f_\omega$ instead of $f_{\omega, 0}$.
	\end{itemize}
	
	The above observations together with Theorem~\ref{1043} easily imply the following lemma.
	\begin{lemma}
		The map
		\[
		\epsilon \mapsto \int_\Omega\int_M f_{\omega,\epsilon}^2(x)dh_{\omega}^\epsilon(x)d\mathbb P(\omega)
		\]
		is differentiable at $\epsilon=0$.
	\end{lemma}
	
	\begin{proof}
		For $\epsilon$ sufficiently close to $0$, let $H(\epsilon)\in L^\infty(\Omega,C^r(M))$ be defined by 
		\[
		H(\epsilon)(\omega)=f_{\omega, \epsilon}^2, \quad \omega \in \Omega.
		\]
		Then, the discussion preceding the statement of the lemma implies that the map $H$ is differentiable at $\epsilon=0$. Now the conclusion of the lemma follows from Theorem \ref{thm:annealedresp} and the simple observation that 
		\[
		\int_\Omega\int_M f_{\omega,\epsilon}^2(x)dh_{\omega}^\epsilon(x)d\mathbb P(\omega)=R(\epsilon, H(\epsilon)).
		\]
		with $R$ given by \eqref{eq:R}.
	\end{proof}

	We recall that (see Subsection~\ref{Pre}) that for $h\in \mathcal B^{p,q}$ and $f\in C^q (M)$ we can define $f\cdot h \in \mathcal B^{p,q}$ whose action as a distribution is given by
	\[
	(f\cdot h)(\phi)=h(f\phi), \quad \text{for $\phi \in C^q(M)$.}
	\]
	Moreover, there exists $C>0$ (depending only on $M$) such that 
	\[
	\lVert f\cdot h\rVert_{p,q}\le C \lVert h\rVert_{p,q} \cdot \lVert f\rVert_{C^q}.
	\]
	The above inequality will be frequently used in what follows and thus we will not explicitly refer to it. Moreover, in what follows, $C>0$ will denote a constant which is independent on all parameters ($\omega$, $n$ etc.) involved. 
	
	Observe that 
	\[( f_{\omega,\epsilon} \cdot  h_{\omega}^\epsilon )(f_{\sigma^n\omega,\epsilon}\circ T_{\omega,\epsilon}^{n})= \L_{\omega,\epsilon}^n (f_{\omega,\epsilon} \cdot h_{\omega}^\epsilon) (f_{\sigma^n\omega,\epsilon}). \]
	In addition,  $( f_{\omega,\epsilon} \cdot h_{\omega}^\epsilon)(1)=h_{\omega}^\epsilon (f_{\omega,\epsilon})=0$.
	We now write
	\begin{equation}
		\frac{1}{\epsilon}\left(\L^{n}_{\omega,\epsilon}(f_{\omega,\epsilon}\cdot h_{\omega}^\epsilon) (f_{\sigma^n\omega,\epsilon} )- \L_\omega^n(f_\omega \cdot h_\omega )(f_{\sigma^n\omega})\right)=(I)_{n,\omega, \epsilon}+(II)_{n,\omega, \epsilon}+(III)_{n,\omega, \epsilon},
	\end{equation}
	where
	\begin{align*}
		(I)_{n,\omega, \epsilon}&:=\L_\omega^n (f_\omega \cdot h_\omega) \bigg ( \frac{1}{\epsilon}\left(f_{\sigma^n\omega,\epsilon}-f_{\sigma^n\omega}\right) \bigg),\\
		(II)_{n,\omega, \epsilon}&:= \frac{1}{\epsilon}\left(\L_{\omega,\epsilon}^n-\L_{\omega}^n\right)(f_{\omega}\cdot h_\omega) (f_{\sigma^n\omega,\epsilon}),\\
		(III)_{n,\omega, \epsilon}&:= \L_{\omega,\epsilon}^n\bigg ( \frac{f_{\omega,\epsilon}\cdot h_{\omega}^\epsilon-f_\omega\cdot  h_\omega}{\epsilon}\bigg )(f_{\sigma^n\omega,\epsilon}).
	\end{align*}

	\begin{lemma}
		For each $n\in \N$, 
		\[
		\lim_{\epsilon \to 0}\esssup_{\omega \in \Omega}\bigg \lvert (I)_{n,\omega, \epsilon}-\hat h_{\sigma^n \omega}(F_{\sigma^n \omega})\L_\omega^n (f_\omega \cdot h_\omega) (1)\bigg \rvert=0.
		\]
		In addition, for $\epsilon$ sufficiently close to $0$, we have  that
		\[
		\esssup_{\omega \in \Omega} \lvert (I)_{n, \omega, \epsilon} \rvert \le Ce^{-\lambda n}.
		\]
	\end{lemma}
	
	\begin{proof}
		The first assertion follows directly from~\eqref{1011}, \eqref{bnd} and~\eqref{TET}. In addition, observe  that for $\epsilon$ sufficiently close to $0$, 
		\[
		\esssup_{\omega \in \Omega}  \bigg \lvert (I)_{n, \omega, \epsilon}- \hat h_{\sigma^n \omega}(F_{\sigma^n \omega})\L_\omega^n (f_\omega \cdot h_\omega) (1)\bigg \rvert  \le Ce^{-\lambda n}.
		\]
		On the other hand, \eqref{1011},  \eqref{sup0} and~\eqref{TET} imply that
		\[
		\esssup_{\omega \in \Omega} \bigg \lvert  \hat h_{\sigma^n \omega}(F_{\sigma^n \omega})\L_\omega^n (f_\omega \cdot h_\omega) (1) \bigg  \rvert \le Ce^{-\lambda n}.
		\]
		The above two estimates readily give the second assertion of the lemma.
	\end{proof}
	
	\begin{lemma}
		For each $n\in \N$, 
		\begin{equation}\label{a23x}
			\lim_{\epsilon \to 0} \esssup_{\omega\in\Omega}\left|(II)_{n,\omega, \epsilon}- \hat\L_{n,\omega}(f_\omega \cdot h_\omega) (f_{\sigma^n\omega})\right|= 0,
		\end{equation}
		where \[\hat\L_{n,\omega}=\sum_{k=1}^n\L_{\sigma^k\omega}^{n-k}\hat\L_{\sigma^{k-1}\omega}\L^{k-1}_{\omega}.\] 
		Furthermore, for $\epsilon$ sufficiently close to $0$, we have that
		\[
		\esssup_{\omega \in \Omega}\lvert (II)_{n, \omega, \epsilon} \rvert \le Cne^{-\lambda' n}.
		\]
	\end{lemma}
	
	\begin{proof}
		In order to prove~\eqref{a23x}, we first claim that 
		\begin{equation}\label{252x}
			\left\|\frac{1}{\epsilon}\left(\L_{\omega,\epsilon}^n-\L_\omega^n\right)(f_\omega \cdot h_\omega)-\hat\L_{n,\omega}(f_\omega \cdot  h_\omega)\right\|_w\leq \tilde\alpha(\epsilon),
		\end{equation}
		with $\tilde\alpha(\epsilon)\to 0$ when $\epsilon\to 0$.
		Observe that 
		\begin{align*}
			\frac{1}{\epsilon}\left(\L_{\omega,\epsilon}^n-\L_\omega^n\right)=\sum_{k=1}^n \L_{\sigma^k\omega,\epsilon}^{n-k}\frac{\L_{\sigma^{k-1}\omega,\epsilon}-\L_{\sigma^{k-1}\omega}}{\epsilon}\L^{k-1}_{\omega},
		\end{align*}
		and therefore
		\begin{align*}
			&\frac{1}{\epsilon}\left(\L_{\omega,\epsilon}^n-\L_\omega^n\right)-\hat\L_{n,\omega}=\sum_{k=1}^n\left[\L_{\sigma^k\omega,\epsilon}^{n-k}\frac{\L_{\sigma^{k-1}\omega,\epsilon}-\L_{\sigma^{k-1}\omega}}{\epsilon}-\L_{\sigma^k\omega}^{n-k}\hat\L_{\sigma^{k-1}\omega}\right]\L^{k-1}_{\omega}\\
			&=\sum_{k=1}^n\left[\left(\L_{\sigma^k\omega,\epsilon}^{n-k}-\L_{\sigma^k\omega}^{n-k}\right)\frac{\L_{\sigma^{k-1}\omega,\epsilon}-\L_{\sigma^{k-1}\omega}}{\epsilon}+
			\L_{\sigma^k\omega}^{n-k}\left(\frac{\L_{\sigma^{k-1}\omega,\epsilon}-\L_{\sigma^{k-1}\omega}}{\epsilon}-\hat\L_{\sigma^{k-1}\omega}\right)\right]\L^{k-1}_{\omega}.
		\end{align*}
		By the arguments in the proof of Proposition~\ref{922},  \eqref{1011}, \eqref{2x}  and~\eqref{bnd}, we have that 
		\begin{equation}\label{pa1}
			\begin{split}
				&\left\|\left(\L_{\sigma^k\omega,\epsilon}^{n-k}-\L_{\sigma^k\omega}^{n-k}\right)\frac{\L_{\sigma^{k-1}\omega,\epsilon}-\L_{\sigma^{k-1}\omega}}{\epsilon}\L^{k-1}_{\omega}(f_\omega \cdot h_\omega)\right\|_w\\
				&\leq C|\epsilon|(n-k)\left\|\frac{\L_{\sigma^{k-1}\omega,\epsilon}-\L_{\sigma^{k-1}\omega}}{\epsilon}\L^{k-1}_{\omega}(f_\omega \cdot  h_\omega)\right\|_s\\
				&\leq C|\epsilon|(n-k)e^{-\lambda(k-1)}\esssup_{\omega\in\Omega}\|f_\omega \cdot h_\omega\|_{ss} \\
				&\le C|\epsilon |(n-k)e^{-\lambda k}.
			\end{split}
		\end{equation}

		Similarly, using \eqref{1011}, \eqref{x2x}, \eqref{lim}, \eqref{3x} and~\eqref{bnd}, we obtain that 
		\begin{equation}\label{pa2}
			\begin{split}
				&\left\|\L_{\sigma^k\omega}^{n-k}\left(\frac{\L_{\sigma^{k-1}\omega,\epsilon}-\L_{\sigma^{k-1}\omega}}{\epsilon}-\hat\L_{\sigma^{k-1}\omega}\right)\L^{k-1}_{\omega}(f_\omega \cdot h_\omega)\right\|_{w}\\
				&\leq Ce^{-\lambda'(n-k)}\left\|\left(\frac{\L_{\sigma^{k-1}\omega,\epsilon}-\L_{\sigma^{k-1}\omega}}{\epsilon}-\hat\L_{\sigma^{k-1}\omega}\right)\L^{k-1}_{\omega}(f_\omega \cdot  h_\omega)\right\|_{w}\\
				&\leq Ce^{-\lambda'(n-k)}\alpha(\epsilon)\|\L^{k-1}_{\omega}(f_\omega \cdot  h_\omega)\|_{ss}\\
				&\leq Ce^{-\lambda'n}\alpha(\epsilon)\esssup_{\omega\in\Omega}\|f_\omega \cdot  h_\omega\|_{ss} \\
				&\le C\alpha(\epsilon)e^{-\lambda' n}.
			\end{split}
		\end{equation}
		Then, \eqref{pa1} and~\eqref{pa2} imply~\eqref{252x}.  
		
		Furthermore, \eqref{dec}, \eqref{1011}, \eqref{x2x}, \eqref{qx} and~\eqref{bnd} imply that
		\begin{equation}\label{pa3}
			\begin{split}
				\lVert  \hat\L_{\omega,n}(f_\omega \cdot h_\omega) \rVert_w &\leq \sum_{k=1}^n\|\L_{\sigma^k\omega}^{n-k}\hat\L_{\sigma^{k-1}\omega}\L^{k-1}_{\omega}(f_\omega \cdot  h_\omega)\|_{s}\\
				&\leq C\sum_{k=1}^n e^{-\lambda(n-k)}\esssup_{\omega\in\Omega} (\|\hat\L_{\sigma^{k-1}\omega}\|_{B_{ss}\to B_s} \cdot \|\L^{k-1}_{\omega}(f_\omega \cdot  h_\omega)\|_{ss})\\
				&\leq C\sum_{k=1}^n e^{-\lambda(n-k)}e^{-\lambda(k-1)}\esssup_{\omega\in\Omega}\|f_\omega \cdot  h_\omega\|_{ss}\\
				&\leq Cne^{-\lambda n}.
			\end{split}
		\end{equation}
		Using Theorem~\ref{sst},  \eqref{bnd}, \eqref{252x} and~\eqref{pa3}, we have that 
		\[
		\begin{split}
			& \esssup_{\omega\in\Omega}\left|(II)_{n,\omega, \epsilon}- \hat\L_{n,\omega}(f_\omega \cdot h_\omega) (f_{\sigma^n\omega})\right| \\
			&\le \esssup_{\omega \in \Omega}\bigg |\frac{1}{\epsilon}\left(\L_{\omega,\epsilon}^n-\L_\omega^n\right)(f_\omega \cdot h_\omega)(f_{\sigma^n \omega, \epsilon})-\hat\L_{n,\omega}(f_\omega \cdot  h_\omega)(f_{\sigma^n \omega, \epsilon})\bigg | \\
			&\phantom{\le}+\esssup_{\omega \in \Omega} \bigg | \hat\L_{n,\omega}(f_\omega \cdot  h_\omega) (f_{\sigma^n \omega, \epsilon}-f_{\sigma^n \omega})\bigg |\\
			&\le  \tilde\alpha(\epsilon)\esssup_{\omega \in \Omega}\lVert f_{\sigma^n \omega, \epsilon}\rVert_{C^r}+Cne^{-\lambda n} \esssup_{\omega \in \Omega} | (h_\omega^\epsilon -h_\omega)(F_\omega)| \\
			&\le C\tilde \alpha (\epsilon)+Cn e^{-\lambda n} |\epsilon | |\log (|\epsilon |)|,
		\end{split}
		\]
		which implies the first assertion of the lemma.
		
		On the other hand, using~\eqref{3x} (which also persists under small perturbations),  \eqref{x2x}, \eqref{qx} and~\eqref{bnd}, we have that for $\epsilon$ sufficiently small, 
		\begin{equation}\label{o1}
			\esssup_{\omega \in \Omega}\left\|\left(\L_{\sigma^k\omega,\epsilon}^{n-k}-\L_{\sigma^k\omega}^{n-k}\right)\frac{\L_{\sigma^{k-1}\omega,\epsilon}-\L_{\sigma^{k-1}\omega}}{\epsilon}\L^{k-1}_{\omega}(f_\omega \cdot h_\omega)\right\|_w \le Ce^{-\lambda'n}.
		\end{equation}
		Moreover, from~\eqref{pa2} it follows that for $\epsilon$ sufficiently small,
		\begin{equation}\label{o2}
			\esssup_{\omega \in \Omega}\left\|\L_{\sigma^k\omega}^{n-k}\left(\frac{\L_{\sigma^{k-1}\omega,\epsilon}-\L_{\sigma^{k-1}\omega}}{\epsilon}-\hat\L_{\sigma^{k-1}\omega}\right)\L^{k-1}_{\omega}(f_\omega \cdot h_\omega)\right\|_{w} \le Ce^{-\lambda' n}.
		\end{equation}
		By~\eqref{o1} and~\eqref{o2}, we have that for sufficiently small $\epsilon$, 
		\[
		\esssup_{\omega \in \Omega}\left\|\frac{1}{\epsilon}\left(\L_{\omega,\epsilon}^n-\L_\omega^n\right)(f_\omega \cdot h_\omega)-\hat\L_{n,\omega}(f_\omega \cdot  h_\omega)\right\|_w \le Cne^{-\lambda' n}.
		\]
		The above estimate together with~\eqref{pa3} easily implies that the second assertion of the lemma also holds. 
	\end{proof}

	By using similar arguments, one can establish the following lemma.
	\begin{lemma}
		For each $n\in \N$,
		\[
		\lim_{\epsilon \to 0} \esssup_{\omega \in \Omega} \bigg \lvert (III)_{n,\omega, \epsilon} -  \L_{\omega}^n(\hat h_\omega (F_\omega)  h_\omega+f_\omega \cdot \hat h_\omega) (f_{\sigma^n\omega})\bigg \rvert =0.
		\]
		Moreover, for $\epsilon$ sufficiently small, we have that
		\[
		\esssup_{\omega \in \Omega}\lvert (III)_{n, \omega, \epsilon} \rvert \le Ce^{-\lambda' n}.
		\]
	\end{lemma}
	The conclusion of Theorem~\ref{variance} follows from  previous lemmas and the dominated convergence theorem. 
	
	\begin{remark}
		In~\cite{DH} the authors have extended the results from~\cite{DFGTV} to the case of vector-valued observables. In particular, the quenched version of the central limit theorem for vector-valued observables was established. In this setting, the variance is a symmetric matrix which is in general positive semi-definite (for the central limit theorem to hold it needs to be positive-definite). One can easily establish the version of Theorem~\ref{variance} in this setting, essentially by repeating the arguments in the proof of Theorem~\ref{variance} for each matrix  component. 
	\end{remark}
	
	\section{Application to other type of random systems.} \label{sec:other}
	In this paper, we focused our efforts on studying (quenched) statistical stability and linear response for random compositions of Anosov diffeomorphisms. Nevertheless, our approach, or a slight modification thereof, is applicable to other types of random hyperbolic systems.
	\subsection{Random uniformly expanding dynamics.} In this subsection, let us describe the application of Theorems \ref{sst} and \ref{1043} to a simple class of fiberwise perturbations of random compositions of uniformly expanding maps of the unit circle $\mathbb S^1$. The setting is close to \cite[\S 6]{GS}: consider a family $(D_\epsilon)_{\epsilon\in I}$ of diffeomorphisms of $\mathbb S^1$ (where $0\in I\subset \R$ is an interval), satisfying\[D_\epsilon= \Id+\epsilon S,\]
	where $S:\mathbb S^1\to \R$ is  a $C^{4}$ map. Letting $(\Omega,\mathcal F,\mathbb P)$ be a probability space, endowed with an invertible, measure-preserving and ergodic map $\sigma:\Omega\circlearrowleft$. We consider a measurable map $\omega\in\Omega\mapsto T_{\omega}\in C^4(\mathbb S^1,\mathbb S^1)$ such that:
	\begin{enumerate}
		\item there exists $\lambda >1$ such that for $\mathbb P$-a.e. $\omega \in \Omega$, $\inf_{x\in\mathbb S^1}|T'_\omega(x)|\ge \lambda$;
		\item $\esssup_{\omega\in\Omega}\|T_\omega\|_{C^4}\le\Delta$ for some small $\Delta>0$.
	\end{enumerate}
	 We then set \[T_{\omega,\epsilon}:= D_\epsilon\circ T_\omega \quad \text{for $\epsilon \in I$ and $\omega \in \Omega$,}\] and we review the assumptions of Theorems \ref{sst} and \ref{1043} for the spaces $\B_{ss}=W^{3,1}(\mathbb S^1)$, $\B_s=W^{2,1}(\mathbb S^1)$ and $\B_w=W^{1,1}(\mathbb S^1)$:
	\begin{itemize}
		\item \eqref{1} and \eqref{3} are established in \cite[\S 5]{C}.
		\item \eqref{2} follows from \cite[Prop. 35]{GS}
		\item By applying~\cite[Prop. 2.10]{CR} (provided that $\Delta$ is sufficiently small),  we conclude that  \eqref{dec} holds on $(\B_{ss},\B_s,\B_w)$.
		\item To define the derivative operator $\hat\L_\omega$, we start by remarking that since $\L_{\omega,\epsilon}=\L_{D_\epsilon}\L_\omega$, one has (see~\cite[Eq. (51)]{GS}) that 
		\[\hat\L_\omega=\left[\frac{d\L_{D_\epsilon}}{d\epsilon}\right] \bigg \rvert_{\epsilon=0}\L_\omega=-(\L_\omega(\cdot)S)'.\]
		It is easy to see that $\hat\L_\omega$ defines a bounded operator from $\B_{ss}$ to $\B_s$ (resp. from $\B_s$ to $\B_w$) and satisfies \eqref{qx}.
		\\As for condition~\eqref{lim}, we have for $\phi\in\B_{s}$
		\begin{align*}
			\left\|\epsilon^{-1}(\L_{\omega,\epsilon}-\L_\omega)\phi-\hat\L_\omega\phi\right\|_w
			&\le \left\|\epsilon^{-1}(\L_{D_\epsilon}-\Id)+(\cdot S)'\right\|_{\B_{s}\to\B_w}\esssup_{\omega\in\Omega}\|\L_\omega\phi\|_{s}\\
			&\le C\alpha(\epsilon)\|\phi\|_{s}
		\end{align*}
		by using \eqref{1}, with $\alpha(\epsilon)=\left\|\epsilon^{-1}(\L_{D_\epsilon}-\Id)+(\cdot S)'\right\|_{\B_s\to\B_w}$,  which goes to $0$ as $\epsilon\to 0$ by \cite[Prop. 36]{GS}.
	\end{itemize}
	
	\subsection{Random piecewise hyperbolic dynamics.} Let us discuss the application of Theorem \ref{sst} to random compositions of close-by piecewise hyperbolic maps, defined on a two-dimensional compact Riemann manifold $X$, as described in \cite[\S10]{DFGTV} and \cite[\S2]{DL}. 
	It is noteworthy that one \emph{cannot} directly apply Theorem \ref{sst}, since, as noted in \cite[\S10.2.1]{DFGTV}, the transfer operator map $\omega\mapsto\L_\omega$ is not strongly measurable. Still, the conclusion of Theorem \ref{sst} hold; let us explain why.
	\par\noindent In \cite[\S2.4]{DL}, the set $\Gamma_A$ of maps $T$ satisfying the assumptions of \cite[\S2]{DL}, with second derivative $|D^2T|<A$ is introduced, as well as the distance $\gamma$ between two such maps. 
	\par\noindent Let us fix a (small enough) $\epsilon_0>0$, a $T\in\Gamma_A$ and let $X_{\epsilon_0}:=\{S\in\Gamma_A: \gamma(T,S)<\epsilon_0\}$. We let $\B_s$ and $\B_w$ be the Banach spaces defined in \cite[\S2.2]{DL}(where $\B_s$ is denoted $\B$). In particular, we recall that elements of $\B_s$ are distributions of order at most $1$. 
	Letting  $I:=\left[-\epsilon_0/2,\epsilon_0/2\right]$, we set, for a fixed $L>0$:
	\[B_{\epsilon_0,L}:=\{\mathcal T:I\to X_{\epsilon_0},~\gamma(\mathcal T(\epsilon),\mathcal T(\epsilon'))\le L|\epsilon-\epsilon'|,~\forall\epsilon,\epsilon'\in I\}.\]
	This can be viewed as a ball of Lipschitz (with respect to the distance $\gamma$) curves from $I$ to $X_{\epsilon_0}$. We now consider a measurable, countably-valued mapping $\mathbf{T}:\Omega\to B_{\epsilon_0,L}$. As before, $(\Omega,\mathcal F,\mathbb P)$ is a probability space endowed with an invertible, measure-preserving and ergodic map $\sigma$ and we will denote $T_{\omega,\epsilon}:=\mathbf{T}(\omega)(\epsilon,\cdot)$.
	\\We claim that for any $\epsilon\in I$, there exists 
	a measurable family $(h_\omega^\epsilon)_{\omega \in \Omega} \subset  \B_s$ such that $\L_{\omega,\epsilon}h_{\omega}^\epsilon=h_{\sigma\omega}^\epsilon$ for $\mathbb P$-a.e. $\omega \in \Omega$ and
	\[\esssup_{\omega\in\Omega}\|h_\omega^\epsilon-h_\omega\|_{w}\le C\epsilon^\beta|\log(|\epsilon|)|,\]
	for some $C>0$, $0<\beta<1$ independent on $\omega$ and $\epsilon$, with $h_\omega:=h_\omega^0$. 
	\\Let us review the assumptions for Theorem \ref{sst} in this context:
	\begin{itemize}
		\item \eqref{dec} holds by \cite[Eq (70)]{DFGTV}, where $\psi \in \B_s'$ is given by $\psi(h)=h(1)$, $h\in \B_s$. 
		\item Up to shrinking $\epsilon_0$, we have \eqref{1} and \eqref{3} by \cite[Eq (71)]{DFGTV}.
		\item Up to replacing $\epsilon$ by $\epsilon^\beta$, \eqref{2} follows from the definition of $B_{\epsilon_0,L}$ and \cite[Lemma 6.1]{DL}.
		\item As usual, \eqref{4} holds as $\L_{\omega,\epsilon}$ is a transfer operator associated to $T_{\omega,\epsilon}$.
	\end{itemize}
	In particular, Proposition \ref{922} (uniform in $\epsilon$ and $\omega$ exponential decay of correlations) holds in the present setting. We cannot use here the fixed-point construction of Proposition \ref{836}, since we do not know whether the cocycle of transfer operators $(\mathcal L_{\omega, \epsilon})_{\omega \in \Omega}$ is strongly measurable. However, we can use \eqref{1}, \eqref{3} and that, for each $\epsilon\in I$, $T_{\omega,\epsilon}$ is countably-valued to apply the  version of the MET for the  so-called $\mathbb P$-continuous cocycles (see~\cite[Theorem 17]{FLQ}): this gives us, as in Remark \ref{rem:MET}, that for each $\epsilon\in I$ there exists:
	\begin{itemize} \item  $1\le l=l(\epsilon)\le \infty$ and a sequence of \emph{exceptional
			Lyapunov exponents}
		\[0=\Lambda (\epsilon)=\lambda_1(\epsilon)>\lambda_2(\epsilon)>\ldots>\lambda_l(\epsilon)>\kappa (\epsilon)\] or in the case
		$l=\infty$, 
		\[0=\Lambda (\epsilon)=\lambda_1(\epsilon)>\lambda_2(\epsilon)>\ldots \quad  \text{with
			$\lim_{n\to\infty} \lambda_n(\epsilon)=\kappa (\epsilon)$};
		\]
		\item a full-measure set $\Omega_\epsilon$ such that for each $\omega\in\Omega_\epsilon$, there is a unique measurable Oseledets splitting
		$$\mathcal{B}_s=\left(\bigoplus_{j=1}^l Y_j^\epsilon(\omega)\right)\oplus V^\epsilon(\omega),$$
		where each component of the splitting is equivariant under $\mathcal{L}_{\omega, \epsilon}$, that is,
		$\mathcal L_{\omega, \epsilon}(Y_j^\epsilon(\omega))= Y_j^\epsilon(\sigma\omega)$ and
		$\mathcal L_{\omega, \epsilon}(V^\epsilon(\omega))\subset V^\epsilon(\sigma\omega)$. The subspaces $Y_j^\epsilon(\omega)$ are finite-dimensional and for each $y\in Y_j^\epsilon(\omega)\setminus\{0\}$,
		\[\lim_{n\to\infty}\frac 1n\log\|\mathcal L_{\omega, \epsilon}^n y\|_s=\lambda_j(\epsilon).\]
		Moreover, for $y\in V(\omega)$, $\lim_{n\to\infty}\frac 1n\log\|\mathcal
		L_{\omega, \epsilon}^n y\|_s\le \kappa(\epsilon)$.
	\end{itemize}
	It follows easily from Proposition~\ref{922} (see the proof of~\cite[Proposition 3.6.]{DFGTV}) that $Y_1^\epsilon (\omega)$ is one-dimensional: for each $\epsilon \in I$, we may thus consider a generator $h_{\omega}^\epsilon$, normalized by $\psi(h_\omega^\epsilon)=1$,  which satisfies $\L_{\omega,\epsilon}h_\omega^\epsilon=h_{\sigma\omega}^\epsilon$. We now claim that 
	\begin{equation}\label{TZ}
		\sup_{\epsilon\in I}\esssup_{\omega\in\Omega}\|h_\omega^\epsilon\|_s<+\infty.
	\end{equation}
	In order to establish~\eqref{TZ}, we start by observing that using~\eqref{dec1} we have that 
	\begin{equation}\label{tz1}
		\| h_\omega^\epsilon-\L_{\sigma^{-n}\omega,\epsilon}^n 1\|_s=\|\L_{\sigma^{-n}\omega,\epsilon}^n(h_{\sigma^{-n} \omega}^\epsilon-1)\|_s \le D'e^{-\lambda' n} \|h_{\sigma^{-n} \omega}^\epsilon-1\|_s,
	\end{equation}
	for $n\in \N$, $\omega \in \Omega$ and $\epsilon \in I$. Furthermore, since $\lambda_1(\epsilon)=0$, we have that the random variable $\omega \mapsto \|h_\omega^\epsilon\|_s$ is tempered\footnote{We recall that a random variable $K\colon \Omega \to (0, +\infty)$ is tempered if $\lim_{n\to \pm \infty} \frac 1 n \log K(\sigma^n \omega)=0$ for $\mathbb P$-a.e. $\omega \in \Omega$.} for each $\epsilon \in I$. Hence, by~\cite[Proposition 4.3.3]{Arnold} for each $\epsilon \in I$, there exists a random variable $K_\epsilon \colon \Omega \to (0, +\infty)$ such that 
	\begin{equation}\label{tz2}
		\|h^\epsilon_\omega -1\|_s \le K_\epsilon (\omega) \quad \text{and} \quad K_\epsilon (\sigma^n \omega) \le e^{\frac{\lambda' |n|}{2} }K_\epsilon (\omega),
	\end{equation}
	for $\mathbb P$-a.e. $\omega \in \Omega$ and $n\in \mathbb Z$. By~\eqref{tz1} and~\eqref{tz2}, we obtain that 
	\[
	\| h_\omega^\epsilon-\L_{\sigma^{-n}\omega,\epsilon}^n 1\|_s \le D'K_\epsilon (\omega)e^{-\frac{\lambda' n}{2}} \quad \text{for $\mathbb P$-a.e. $\omega \in \Omega$ and $n\in \N$,}
	\]
	which implies that for $\epsilon \in I$,
	\begin{equation}\label{tz3}
		h_\omega^\epsilon=\lim_{n\to \infty} \L_{\sigma^{-n}\omega,\epsilon}^n 1 \quad \text{in $\B_s$, for $\mathbb P$-a.e. $\omega \in \Omega$.}
	\end{equation}
	Clearly, \eqref{TZ} follows readily from~\eqref{1} and~\eqref{tz3}. From there, we can reproduce the proof of Theorem \ref{sst}, to get the announced result.
	
	\begin{remark}
		It is natural to ask whether 
		Theorem \ref{1043} can be applied  in the piecewise hyperbolic setting described above. First, we note that there is no natural candidate for a $\B_{ss}$ space. Indeed,  as noticed in the introduction of \cite{DL} (and in contrast with the situation in \cite{GL}),  considering a (piecewise) $C^r$ or a (piecewise) $C^s$, $r>s>2$ system yields the same couple ($\B_w$ and $\B_s$) of Banach spaces. In other words, the degree of the smoothness of maps doesn't influence the construction of the anisotropic spaces, which makes unclear whether this line of reasoning can produce a space $\B_{ss}$ satisfying our requirements. 
		In fact, to the best of our knowledge there are currently no results dealing with the linear response for classes of piecewise hyperbolic dynamics described above even in the deterministic setting (i.e. when we take  $\Omega$ to be a singleton).
		\\Secondly, the case of deterministic, one-dimensional piecewise expanding maps \cite{B1,BS1} suggests that, in general, linear response does not hold in a piecewise smooth setting.
		\\Finally, we notice that for random compositions of billiard maps such as described, e.g. in \cite{DZ} do not fall under the setup of Theorem \ref{sst}, as they do not satisfy Lasota-Yorke inequalities of the type \eqref{1} and \eqref{3} (the $\|.\|_w$ carries a factor $\eta^n$ for some $\eta\ge 1$).
	\end{remark}
	
	\section{Acknowledgments}
We would like to thank the anonymous referee for his/hers constructive comments.

	D.D. was supported in part by Croatian Science Foundation under the project
	IP-2019-04-1239 and by the University of Rijeka under the projects uniri-prirod-18-9 and uniri-pr-prirod-19-16.

	J.S. was supported by the European Research Council (ERC) under the European Union's Horizon 2020 research and innovation programme (grant agreement No 787304).
	
	\bibliographystyle{amsplain}

\end{document}